\documentclass[10pt]{article}
\usepackage{amsmath, amsfonts, amsthm, amssymb, verbatim}%,dsfont}
\usepackage{graphicx}
\usepackage[mathcal]{euscript}
\usepackage{amstext}
\usepackage{amssymb,mathrsfs}
\usepackage[ansinew]{inputenc}
\newcommand {\HH}     {\boldsymbol{H}}
\newcommand {\BB}     {\boldsymbol{B}}
\newcommand {\EE}     {\boldsymbol{E}}
\newcommand {\DD}     {\boldsymbol{D}}
\newcommand {\JJ}     {\boldsymbol{J}}
\newcommand {\FF}     {\boldsymbol{F}}
\newcommand {\UU}     {\boldsymbol{U}}
\newcommand {\hh}     {\boldsymbol{h}}
\newcommand {\jj}     {\boldsymbol{j}}
\newcommand {\ppsi}   {\boldsymbol{\psi}}
\newcommand{\medint}{{\mbox{\vrule height3.5pt depth-2.8pt
         width6pt}\mkern-17mu\int\nolimits}}
\newcommand{\R}{\mathbb R}
\newcommand{\T}{\mathbb T}

\newcommand{\Z}{\mathbb Z}
\newcommand{\nat}{\mathbb N}

\newcommand \be     {\begin{equation}}
\newcommand \ee     {\end{equation}}

\newcommand \del     \partial

\newcommand{\mbf}[1]{{\mathbf{#1}}}
\newcommand{\clg}[1]{{\mathcal{#1}}}

\def\XXint#1#2#3{{\setbox0=\hbox{$#1{#2#3}{\int}$}
\vcenter{\hbox{$#2#3$}}\kern-.5\wd0}}

\DeclareMathOperator\dive {div}

\DeclareMathOperator\tr {tr} 

\newcommand{\curl}{\,{\rm curl}}

\newcommand{\iinte}{\int \! \! \! \int}

\newtheorem{theorem}{Theorem}[section]

 \newtheorem{remark}[theorem]{Remark}
\newtheorem{lemma}[theorem]{Lemma}
\newtheorem{corollary}[theorem]{Corollary}
\newtheorem{definition}[theorem]{Definition}

 \def\beqs{\begin{eqnarray*}}
 \def\enqs{\end{eqnarray*}}
 \def\beq{\begin{eqnarray}}
 \def\enq{\end{eqnarray}}

\def\bfh{\mbox{\boldmath$h$}}

\def\bfj{\mbox{\boldmath$j$}}
\def\bfa{\mbox{\boldmath$a$}}
\def\bfe{\mbox{\boldmath$e$}}
\def\bfv{\mbox{\boldmath$v$}}
\begin{document}
\title{Well-posedness of the Cauchy Problem on Torus to Electromagnetoelastic System}
\author{Wladimir Neves$^1$, Viatcheslav Priimenko$^2$, Mikhail Vishnevskii$^3$}

\date{}

\maketitle

\footnotetext[1]{ Institute of Mathematics, Federal University of
Rio de Janeiro, C.P. 68530, Cidade Universit\'aria, 21945-970, Rio
de Janeiro, Brazil. E-mail: {\sl wladimir@im.ufrj.br.}}
\footnotetext[2] {Laboratory of Petroleum Engineering and Exploration, North Fluminense State University Darcy Ribeiro,
Rod. Amaral Peixoto, km 163, Imboacica, Maca\'e, RJ, 27925-310, Brazil. E-mail: {\sl slava@lenep.uenf.br.}}
\footnotetext[3]{ Laboratory of Computational Mathematics, North Fluminense State University Darcy Ribeiro, av. Alberto Lamego, 2000
Campos dos Goytacazes, RJ, 28013-602, Brazil. E-mail: {\sl mikhail@uenf.br.}
\newline
\textit{To appear in:}
%Communications of Partial Differential Equations. \textit{AMS
%Subject Classification.} {Primary: . Secondary: .}
\newline
\textit{Key words and phrases.} Electromagnetoelastic coupling;
Nonlinear model; Cauchy problem on torus; Well-posedness.}

%%%%%%%%%%%%%%%%%%%%%%%%%%%%%%%%%%%%%%%%%%%%%%%%%%%%%%%%%%%%%%%%%
%%%%%%%%%%%%%%%%%%%%%%%%%%%%%%%%%%%%%%%%%%%%%%%%%%%%%%%%%%%%%%%%%
%%%%%%%%%%%%%%%%%%%%%%%%%%%%%%%%%%%%%%%%%%%%%%%%%%%%%%%%%%%%%%%%%
%
%\setlength{\baselineskip}{1.3\baselineskip}
%
%%%%%%%%%%%%%%%%%%%%%%%%%%%%%%%%%%%%%%%%%%%%%%%%%%%%%%%%%%%%%%%%%%%%%%%%%%%%
\begin{abstract}
We prove the well-posedness of the Cauchy problem on torus to an
eletromagnetoelastic system. The physical model consists of three
coupled partial differential equations, one of them is a
hyperbolic equation describing the elastic medium and two other
ones form a parabolic system, which comes from Maxwell's
equations. Experimental measurements suggest that the elastic
medium has a periodic structure, moreover with finite number of
discontinuities on the fundamental domain. Thus we have study in
this paper the problem which we have defined as periodically
Cauchy diffraction problem.
\end{abstract}
%%%%%%%%%%%%%%%%%%%%%%%%%%%%%%%%%%%%%%%%%%%%%%%%%%%%%%%%%%%%%%%%%%%%%%%%%%%%
%

\maketitle

\tableofcontents

%%%%%%%%%%%%%%%%%%%%%%%%%%%%%%%%%%%%%%%%%%%%%%%%%%%%%%%%%%%%%%%%%%%%%%%%%%%%%%%%%%%%%%%%%%%%%%%%%%%%%%%%%
\section{Introduction} \label{s.1}
%%%%%%%%%%%%%%%%%%%%%%%%%%%%%%%%%%%%%%%%%%%%%%%%%%%%%%%%%%%%%%%%%%%%%%%%%%%%%%%%%%%%%%%%%%%%%%%%%%%%%%%%%

In this paper we provide a general framework leading to
electromagnetoelastcity theory. In particular, this general theory
encompass the two most important models given by a quasilinear and
a semi-linear system of partial differential equations, described
respectively by equations (see below) \eqref{EME}, \eqref{EMNLCR},
\eqref{NLCRS}, \eqref{LEMEE}, \eqref{6} and, \eqref{EME},
\eqref{EMLCR}, \eqref{LCRS}, \eqref{LEMEE}, \eqref{6}. Moreover,
considering the semi-linear case, and standing to plane waves, we
have proved existence of weak solution, uniqueness and stability,
therefore established the well-posedness of the periodically
Cauchy diffraction problem, see Section \ref{IVP}. In fact, it
seems that the notion of periodically Cauchy diffraction problem
is one of the most realistic one as geophysical experiments
suggest, that is, the medium has a natural periodicity with a
finite number of discontinuities on the fundamental domain. Then,
it is more physical correct to assume that the medium is periodic
in space.

\bigskip The mathematical electromagnetoelasticity theory describes
the interacting effects of an elastic solid medium and an
electromagnetic field applied on it. Thus deformations
experimented by the elastic solid are due to external
electromagnetic forces. More precisely, if an elastic
electroconductive medium is imbedded in an electromagnetic field,
then the elastic waves propagating through the medium will excite
oscillations of the electromagnetic field and themselves will
change under influence of the latter. Moreover, the waves which
arise as a result of such an interaction are called as { \em
electromagnetoelastic waves\/}. We stress that, the first attempts
to apply the theory of electromagnetoelasticity to the
investigation of the wave propagation process in electroconductive
media were made by Knopoff \cite{K}, Chadwick \cite{C}, Dunkin and
Eringen \cite{DE}.

\bigskip
In fact, because of the importance of the applications, in
particular to geophysics applied in seismology, the petroleum
reservoir research, the theory of electromagnetoelasticity have
been developed fast recently. We address  some of the important
mathematical works on the propagation of electromagneticelastic
waves, those are: Avdeev, Goryunov, Soboleva, and Priimenko
\cite{AGSP}, Lorenzi and Priimenko \cite{LoP}, Lorenzi and Romanov
\cite{LR}, Priimenko and Vishnevskii \cite{PV1, PV2}, Romanov
\cite {R1,R2}.
%
%\bigskip
%%%%%%%%%%%%%%%%%%%%%%%%%%%%%%%%%%%%%%%%%%%%%%%%%%%%%%%%%%%%%%%%%%%%%%%%%%%%%%%%%%%%%%%%%%%%%%%%%%%
%\subsection{Purpose and results}
%%%%%%%%%%%%%%%%%%%%%%%%%%%%%%%%%%%%%%%%%%%%%%%%%%%%%%%%%%%%%%%%%%%%%%%%%%%%%%%%%%%%%%%%%%%%%%%%%%

\bigskip
%%%%%%%%%%%%%%%%%%%%%%%%%%%%%%%%%%%%%%%%%%%%%%%%%%%%%%%%%%%%%%%%%%%%%%%%%%%%%%%%%%%%%%%%%%%%%%%%%%%%%%%%%%%%%%%%%%%%%%%%%%%%%%%%%%%%%%%%%%%%%%%
\section{Non-linear electromagnetoelasticity
theory}\label{ss.1.1}
%%%%%%%%%%%%%%%%%%%%%%%%%%%%%%%%%%%%%%%%%%%%%%%%%%%%%%%%%%%%%%%%%%%%%%%%%%%%%%%%%%%%%%%%%%%%%%%%%%%%%%%%%%%%%%%%%%%%%%%%%%%%%%%%%%%%%%%%%%%%%%%

The models of electromagnetism and elasticity could be given in a
rational continuum physics way through the respective stored
energy functions. First, let us see the electromagnetism theory,
where we follow the Coleman and Dill work, see \cite{CD}. Let
$(t,x) \in \R \times \R^3$ be the time-space domain and, we
consider the magnetic intensity field $\HH$, the electric
intensity field $\EE$, the magnetic induction $\BB$, the electric
induction $\DD$, all of them taking values in $\R^3$. Moreover, we
consider the electromagnetic store energy function
$\psi(\DD,\BB)$, thus we have
\begin{equation} \label{EME}
\begin{aligned}
 \partial_t \DD - \curl_x (\partial_{\BB} \psi)&= - \JJ, \\[5pt]
 \partial_t \BB + \curl_x (\partial_{\DD} \psi)&= 0, \\[5pt]
 \dive_x \DD= \rho_e, \dive_x \BB&= 0,
\end{aligned}
\end{equation}
where $\JJ$ is the induced electric current density field, also
taking values in $\R^3$, and $\rho_e$ is the electric charge
density. Further, we obtain from the stored energy function the
following constitutive relations
$$
  \partial_{D_i} \psi(\DD,\BB)=:E_i, \quad \partial_{B_i} \psi(\DD,\BB)=: H_i \quad \quad (i=1,2,3).
$$
We recall that, the first and second equations in \eqref{EME} are
respectively the Ampere and Faraday's Law, and the last ones are
constrains, which are compatible with the first two ones. In the
linear theory of electromagnetism, called Maxwell's equations, we
have
\begin{equation}
\label{EMLCR}
  \DD= \epsilon \EE \quad \text{and} \quad \BB= \mu_e \HH,
\end{equation}
where $\epsilon, \mu_e$ are respectively the dielectric and
permeability tensors. There are many reasons to avoid the linear
case, for instance the singularity at the origin on the electric
field $\EE$. Therefore, some non-linear models have been proposed
and the most famous one is due to M. Born and L. Infeld, see
\cite{BI}. In this model the electromagnetic stored energy
function is given by
\begin{equation}
\label{EMNLCR}
  \psi(\DD,\BB)= \sqrt{ 1 + |\DD|^2 + |\BB|^2 + |\DD \times \BB|^2}.
\end{equation}

\bigskip

Now, we turn our attention to the elastic non-homogeneous medium.
Let $\UU(t,x)$ be the elastic displacement field taking values in
$\R^3$. For elastic materials, the constitutive relation of the
stress tensor is given by
$$
  \mbf{T}= \mbf{T}(\mbf{F}),
$$
where $\mbf{F}$ is the gradient deformation of the solid,
satisfying $\det \mbf{F}>0$. Moreover, considering isothermal
deformations, which means that, the temperature of the medium is
assumed constant and uniform throughout the entire solid, the
stress tensor $\mbf{T}$ could be written as
\begin{equation}
\label{NLCRS}
  \mbf{T}(\mbf{F}) + \pi \; I_d = \clg{T}(\mbf{F})= \rho \frac{\partial \Psi(\mbf{F})}{\partial \mbf{F}} \mbf{F}^T,
\end{equation}
where $\pi$ is a constant, $\rho$ is the mass function per unit volume of the solid, and $\Psi$ is the elastic stored energy function. Furthermore, we have assumed that the elastic material is incompressible. Moreover, considering that the medium is locally isotropic and standing for linear elasticity, we could write
\begin{equation}
\label{LCRS}
  \clg{T}= \lambda \, \tr(\mbf{S}) \, I_d + 2 \, \mu \, \mbf{S},
\end{equation}
where $\lambda$, $\mu$ are scalar functions called the Lame elastic moduli and $\mbf{S}$ is the infinitesimal strain tensor, given by
$$
  \mbf{S}:= \frac{1}{2}(\nabla \UU + \nabla \UU^T).
$$
We recall that \eqref{LCRS} is also known as Hooke's law.
Therefore, from the Cauchy's first law and the Hooke's law, we
have the following equation describing the evolution of the linear
elastic electrically-magnetic conduction solid
\begin{equation}\label{LEMEE}
    \rho \, \partial^2_{tt}{\UU}= \dive_x \mathcal{T} + \FF^b,
\end{equation}
where $\FF^b:= \JJ \times \BB + \rho_e \EE + \FF^s$ is the body field force per unit mass, more precisely $\JJ \times \BB$ is the electric-magnetic part that retards the motion of the solid, $\rho_e \EE$ stands the part of the body force due to the existence of the charge density $\rho_e$, and finally $\FF^s$ are others body forces, for instance gravitational effects on the solid.

\medskip
Therefore, we have seen that the electromagnetic field influences
the elastic field by entering the elastic stress equations of
motion as a body force called Lorentz's ponderomotive force. Now,
we are going to establish that the elastic field turn influences
in the electromagnetic field by modifying the Ohm's law. Indeed,
in a moving conductor medium the current is determined by Ohm's
law, see \cite{K}, that is
\begin{equation}\label{6}
    \JJ= \sigma \; \big(\EE + \partial_t \UU \times \BB \big)+\rho_e \; \partial_t \UU + \FF^e,
\end{equation}
where $\sigma$ is the electrical conductivity tensor field and $\FF^e$ is an external electromagnetic force. Then, we see that the current distribution is modified by the elastic deformations. Thus the interaction between the elastic field and electromagnetic field is expressed through equations \eqref{LEMEE} and \eqref{6}.

\medskip
Now, we are in position to establish the basis of the
electromagnetoelasticity theory. In fact, equations \eqref{EME},
\eqref{EMLCR}, \eqref{LCRS}, \eqref{LEMEE} and \eqref{6} form a
semi-linear system of partial differential equations and,
\eqref{EME}, \eqref{EMNLCR}, \eqref{NLCRS}, \eqref{LEMEE} and
\eqref{6} a quasi-linear one, which is much more complicated,
since shocks are allowed to exist, see Dafermos \cite{D}, and for
instance, Neves \& Serre \cite{NS} for the nonlinear Maxwell's
equations setting. In this paper we consider the semi-linear case
and in other to solve the problem both mechanical and
electromagnetic data should be given. We stress that even in the
semi-linear setting, the nonlinearity involved is non-trivial and
difficult to hand up.

\medskip
Finally, we observe that the theories of electromagnetoelasticity have been extended in various ways. For more acquaintance with the modern state of the theory, the reader is addressed to, e.g., \cite{B,EM}.

\medskip
An outline of this paper follows. In the rest of this section we fix the notation and give some mathematical definitions, which will be used during the paper. In Section \ref{s.2}, we establish our problem, where we give the exact notion of Periodically Cauchy diffraction problem. Finally, Section \ref{s.3} presents the well-posedness result for the problem proposed at the preceding section.

\bigskip
%%%%%%%%%%%%%%%%%%%%%%%%%%%%%%%%%%%%%%%%%%%%%%%%%%%%%%%%%%%%%%%%%%%%%%%%
\subsection{Functional notation and background} \label{FN}
%%%%%%%%%%%%%%%%%%%%%%%%%%%%%%%%%%%%%%%%%%%%%%%%%%%%%%%%%%%%%%%%%%%%%%%%

At this point we fix some functional spaces, which will be used in the paper. For convenience, we follow the notation of Ladyzhenskaia, Solonnikov and Uralceva, see \cite{LSU}. Although, we will be in the periodic setting, which is not the case in \cite{LSU}. Therefore, we address also \cite{VK} and \cite{SW}.

\medskip
We will concern periodic functions in $\R$ (i.e., in the spacial domain). For simplicity, we take the fundamental period to be one. So, we define $\T:= \R \setminus \Z$ to be the 1-dimensional torus, it means that, a scalar periodic function $f: \T \to \R$, satisfies
$$
  f(x+\kappa)= f(x) \quad \text{for each $x \in \R$ and $\kappa
  \in \Z$}.
$$
Moreover, we recall that a periodic function is completely determined by its values in the fundamental domain, here we take $\Omega:= [\,0,1)$. Therefore, periodic functions on $\R$ will be considered as functions on $\T$ or functions on $\Omega$. In fact, the point of view depends only on the context.

\medskip
By $dz$, we denote the Lebesgue measure on $\R$ and for measure-theoretic purposes, we do not distinguish between the measure induced on $\T$ and the Lebesgue measure, that is, if $f$ is any function on $\T$, we define
$$
  \int_\T f dz:= \int_\Omega f dz
$$
when the right hand side makes sense; indeed a function $f$ on $\T$ is measurable when the corresponding function on $\Omega$ is Lebesgue measurable. The same definition could be obtained with the concept of periodic distributions and a periodic partition of unity, see Vo-Khac Khoan  \cite{VK}, Tome 2.

\medskip
The Banach space $L_q(\T)$ consists of all measurable functions on $\T$ that are $q$th-power ($q\geq1$) summable on $\T$ provided with the norm
$$
  \|v\|_{q,\T}= \big(\int_{\T}|v(z)|^q \, dz \big)^{1/q}.
$$

\medskip
For some $T>0$, we set $\Pi_T:= (0,T) \times \T$ and analogously $Q_T:= (0,T) \times \Omega$. The Banach space
$L_{q,\gamma}(\Pi_T)$, $(q, \gamma\geq 1)$, consists of all measurable functions on $\Pi_T$ with a finite norm
$$
  \|v\|_{q,\gamma,\Pi_T}= \big(\int^T_0(\int_{\T}|u(t,z)|^qd z)^{\frac{\gamma}{q}}d t \big)^{1/\gamma}.
$$
Moreover, when $q=\gamma$ the Banach space $L_{q,q}(\Pi_T)$ will be denoted by $L_q(\Pi_T)$, and the norm $\|v\|_{q,q,\Pi_T}$ - by
$\|v\|_{q,\Pi_T}$.

\medskip
Weak (generalized) derivatives should be understood in the customary way, we address the book of Evans, \cite{E}. For $l$ an integral and $q \ge 1$, we denote by:

\medskip
$W_q^l(\T)$,
% for $l$ integral is
the Banach space consisting of all functions of $L_q(\T)$ having weak derivatives of all forms up to order $l$ inclusively, that
are $q$th-power summable on $\T$. The norm in $W^l_q(\T)$ is defined by the equality
$$
\|v\|_{q,\T}^{(l)}=\sum_{s=0}^l\|\partial^s_z v\|_{q,\T}\,.
$$

\medskip
$\stackrel{o}W^l_q(\T)$, the closure in $W_q^l(\T)$ of all functions that are infinitely differentiable and finite in $\T$.

\medskip
$W^{2l,l}_q(\Pi_T)$, the Banach space consisting of the $L_q(\Pi_T)$-elements having weak derivatives of the form
$D^r_tD^s_z$ with any  $r,s$ satisfying the inequality $2r+s\leq2l$. The norm in it is defined by the equality
$$
\|v\|^{(2l)}_{q,\Pi_T}=\sum_{j=0}^{2l}\sum_{2r+s=j}\|\partial^r_t\partial^s_z v\|_{q,\Pi_T}\,,
$$
where the summation $\sum_{2r+s=j}$ is taken over all nonnegative integers $r$ and $s$ satisfying the condition $2r+s=j$.

\medskip
$W^{1,k}_2(\Pi_T)$, $(k=0,1)$, the Hilbert space endowed with the scalar product
$$
(u,v)_{W^{1,k}_2(\Pi_T)}=\iinte_{\Pi_T}(uv+u_zv_z + k \, u_tv_t) \, dz dt \,.
$$
\medskip
$V_2(\Pi_T)$, the Banach space consisting of all $W^{1,0}_2(\Pi_T)$-elements having a finite norm
$$
|v|_{\Pi_T}=\mbox{vrai}\max_{t\in[0,T]}\|v\|_{2,\T}+\|v_z\|_{2,\Pi_T}\,,
$$
where here and below
$$
\|v_z\|_{2,\Pi_T} = \big(\iinte_{\Pi_T}(v_z)^2 \, dzdt \big)^{1/2}\,.
$$

$V^{1,0}_2(\Pi_T)$, the Banach space obtained by completing the set $W^{1,1}_2(\Pi_T)$ in the norm of $V_2(\Pi_T)$.

\medskip
$V^{1,1/2}_2(\Pi_T)$, the subset of those elements $v \in V^{1,0}_2(\Pi_T)$, such that
$$
\int_0^{T-\tau}\int_{\T}\tau^{-1}(v(z,t+\tau)-v(z,t))^2 \, dzdt \rightarrow0\; \mathrm{as}\,\tau\rightarrow0\,.
$$
\begin{remark}
$C^{\alpha,\alpha/2}(Q_T)$ is the set of all continuous functions in $\overline{Q}_T$ satisfying H\"{o}lder conditions in the spacial variable $z$ with exponent $\alpha$ and in the time variable $t$ with exponent $\alpha/2$. Following Stein \& Weiss \cite{SW}, we remark that $C(\T)$ does not correspond to the class of continuous functions on $\Omega$, but only to those functions which remain continuous on $\R$ when extended by periodicity. Therefore, $C^{\alpha,\alpha/2}(\Pi_T)$ should be understood in this sense. Although, the spaces $L_q(\Omega)$, $L_{q,\gamma}(Q_T)$, etc. are defined in a similar way as above, where we replace respectively $\T$ by $\Omega$ and $\Pi_T$ by $Q_T$.
\end{remark}

\bigskip
%%%%%%%%%%%%%%%%%%%%%%%%%%%%%%%%%%%%%%%%%%%%%%%%%%%%%%%%%%%%%%%%%%%%%%%%%%%%%%%%%%%%%%%%%%%%%%%%%%%%%%%%%%%%%%%%%%%%%%%%%%%%%%%%%%%%%%%%%%%%%%%%%%
\section{Statement of the problem}\label{s.2}
%%%%%%%%%%%%%%%%%%%%%%%%%%%%%%%%%%%%%%%%%%%%%%%%%%%%%%%%%%%%%%%%%%%%%%%%%%%%%%%%%%%%%%%%%%%%%%%%%%%%%%%%%%%%%%%%%%%%%%%%%%%%%%%%%%%%%%%%%%%%%%%%%%

%%%%%%%%%%%%%%%%%%%%%%%%%%%%%%%%%%%%%%%%%%%%%%%%%%%%%%%%%%%%%%%%%%%%%%%%%%%%%%%%%%%%%%%%%%%%%%%%%%%%%
\subsection{Plane waves}
%%%%%%%%%%%%%%%%%%%%%%%%%%%%%%%%%%%%%%%%%%%%%%%%%%%%%%%%%%%%%%%%%%%%%%%%%%%%%%%%%%%%%%%%%%%%%%%%%%%%%

We are going to focus on plane waves, which depend only on the one scalar space variable and time. Let us consider $x_3$, the such spacial coordinate. Therefore, all of the fields involved in equations \eqref{EME}, \eqref{EMLCR}, \eqref{LCRS}, \eqref{LEMEE} and \eqref{6} depend on $(t,0,0,x_3) \equiv (t,x_3)$ variables. Moreover, the vector fields $\FF^e$ and $\FF^s$ have the following representations
$$
   \FF^e=(F^e_1,F^e_2,0),\qquad \FF^s=(0,0,F^s_3),
$$
where $F^e_1$, $F^e_2$ and $F^s_3$ are scalar functions and, from now on, $z$ stands for the variable $x_3$.

In the case of diffusion electromagnetic processes the time derivative of $\DD$ in the Ampere's Law is very small in comparison with the conduction current $\JJ$ and, and in such way could be dropped. Further, we assume that $\rho=\mathrm{const}, \mu_e=\mathrm{const}$, $\sigma$ is a scalar function and the charge density $\rho_e=0$, that is, the media satisfies the quasi-neutrality condition. Then, considering these
assumptions, we have from \eqref{EME}, \eqref{EMLCR}, \eqref{LCRS}, \eqref{LEMEE} and \eqref{6} the following system of semi-linear partial differential equations
\begin{align}
      \partial_t H_1&= \partial_z \big(\frac{1}{\sigma\mu_e} \partial_z H_1
      - H_1 \, \partial_t U - \frac{1}
      {\sigma\mu_e} F_1^e), \label{EQ1} \\[5pt]
      \partial_t H_2&= \partial_z \big(\frac{1}{\sigma\mu_e} \partial_z H_2 - H_2 \, \partial_t U - \frac{1}
      {\sigma\mu_e} F_2^e), \label{EQ2} \\[5pt]
      \partial^2_{tt}{U}&=\partial_z \big(\upsilon^2 \, \partial_z U - \frac{\mu_e}{2\rho}(H_1^2+H_2^2) \big)
      +\frac{1}{\rho}F^s_3, \label{EQ3}
\end{align}
where $\upsilon:= \sqrt{(\lambda+2\mu)/\rho}$ is the longitudinal elastic wave velocity and $U:= U_3$.
For simplicity of notation and similarity purposes, we will form a dimensionless system of equations. Let $L$, $V_0$ and $H_0$ be the
characteristic values of length, velocity and magnetic field, respectively. We set,
$$
  r:= \frac{1}{\mu_e \, L \, V_0 \, \sigma}, \quad   p:= \frac{\mu_e \, H_0^2}{2 \, \rho \, V_0^2}, \quad
  \nu:= \frac{\upsilon}{V_0},
$$
respectively the first and third one, the dimensionless magnetic viscosity and the dimensionless velocity of the elastic waves propagation. Then, after simple transformations, we obtain from \eqref{EQ1}--\eqref{EQ3}
\begin{align}
\label{9}
  \hh_t&= \big( r \, \hh_z - \hh \, u_t - r \jj \big)_z,\\
\label{10}
   u_{tt}&= \big( \nu^2 \, u_z - p \, \hh^2 \big)_z + f,
\end{align}
where $\hh:=(h_1,h_2)$, $\jj:=(j_1,j_2)$, $u$ and $f$ are respectively the dimensionless analogues of
$$
  H_i, F^e_i, \quad(i=1,2), \quad U, \quad \frac{1}{\rho} F^s_3\,,
$$
and $\hh^2\equiv\hh\cdot\hh=h_1^2+h_2^2$.

\bigskip
%%%%%%%%%%%%%%%%%%%%%%%%%%%%%%%%%%%%%%%%%%%%%%%%%%%%%%%%%%%%%%%%%%%%%%%%%%%%%%%%%%%%%%%%%%%%%%
\subsection{Periodically Cauchy diffraction problem} \label{IVP}
%%%%%%%%%%%%%%%%%%%%%%%%%%%%%%%%%%%%%%%%%%%%%%%%%%%%%%%%%%%%%%%%%%%%%%%%%%%%%%%%%%%%%%%%%%%%%%

In this section, we state our initial-value problem on torus, that is, we consider equations \eqref{9}--\eqref{10} in $\Pi_T$. We assume that $p$ is a positive number, and $r(z)$, $\nu(z)$ are positive bounded 1-periodic functions, piecewise smooth and discontinuous at each point $z= z_k + \kappa$, $(k= 1,2,\dots,m; \; m \in \nat)$, $(\kappa \in \Z)$, with
$$
  0 < z_1 <z_2 < \ldots < z_m < 1.
$$
The system \eqref{9}--\eqref{10} of partial differential equations is supplemented by the initial-data
\begin{align}
\label{11}
   \hh= \hh_0 \quad \text{on $\{0\}\times\T $}&,\\
\label{12}
   u=u_0,\; u_t=u_1 \quad \text{on $\{0\}\times\T$}&,
\end{align}
where $\hh_0$, $u_0$ and $u_1$ are given 1-periodic smooth functions.

\medskip
Now, since the coefficients in \eqref{9}--\eqref{10} have $m$ discontinuities in $\Omega$, we are going to say that \eqref{9}--\eqref{12} form a periodically Cauchy diffraction problem, which is defined in analogy with one at Ladyzhenskaia, Solonnikov and Uralceva, see \cite[pp.~224-232]{LSU}. Thus for each $k=0,1,\ldots,m$, we set $\Omega^k:= [\, z_k,z_{k+1})$, with $z_0= 0$ and $z_{m+1}=1$. Therefore, we have
$$
  \Omega= \bigcup_{k=0}^{m} \Omega^k \quad \text{and} \quad Q_T= \bigcup_{k=0}^{m} Q_T^k,
$$
where $Q_T^k= (0,T) \times \Omega^k$. It follows that, in each domain $Q_T^k$ there is given a parabolic-hyperbolic system with smooth coefficients and free terms. One of the main purposes is to find in $\Pi_T$ a (weak) solution $(\hh,u)$ of this system, satisfying:
\begin{itemize}
    \item in $Q^{k}_T, k=1,2,\dots,m$, the corresponding equations \eqref{9}--\eqref{10};
    \item on the lower base of $\Pi_T$ the initial condition \eqref{11}--\eqref{12};
    \item at the jump points  $z \equiv z_k + \kappa$, the following compatibility conditions,
   \begin{align}
      \label{15}
      &\left[\bfh\right]=\mathbf{0},\quad\left[u\right]=0,\\
      \label{16}
      &\left[r(\hh_z-\bfj)\right]=\mathbf{0},\quad\left[\nu^2u_z\right]=0,
  \end{align}
\end{itemize}
where $\left[v\right]$ denotes the jump of the function $v$ as it passes through a discontinuous point.

\bigskip
\begin{remark}
\label{RIBVP} Here, we observe our strategy to show existence of solution to periodically Cauchy diffraction problem \eqref{9}--\eqref{16}. First, we are going to consider the associated initial-boundary value problem given by the equations \eqref{9}--\eqref{10} posed in $Q_T$, the initial-data \eqref{11}--\eqref{12} on $\{0\} \times \Omega$ and the jump conditions \eqref{15}--\eqref{16}. Moreover supplemented with the
following boundary conditions
\begin{align}
\label{BC1}
  \hh(t,0)= \hh(t,1), \quad
  \hh_z(t,0)= \hh_z(t,1), \\[5pt]
  \label{BC2}
  u(t,0)= u(t,1), \quad
  u_z(t,0)= u_z(t,1).
\end{align}
Once we have constructed the existence of solution to \eqref{9}--\eqref{BC2} in $\clg{D}'(Q_T)$, we extent it to $\Pi_T$ by periodicity and show that it is a solution to the Cauchy problem \eqref{9}--\eqref{16} in $\Pi_T$, as the definition given below.
\end{remark}

\bigskip
The following definition tells us in which sense a pair $(\hh,u)$ is a weak solution to (\ref{9})--(\ref{16}).

\begin{definition}
\label{def1} A pair of 1-periodic functions in the $z$-spacial variable
$$
  \Big(\hh \in V_{2}(\Pi_T), \quad u \in W^{1,1}_{2}(\Pi_T) \Big)
$$
is called a weak solution of the initial-value problem \eqref{9}--\eqref{16} if satisfies the identities
\begin{align}
   &- \iinte_{\Pi_T} h_l \, \eta_t \, d z d t + \iinte_{\Pi_T} r \, h_{lz} \, \eta_z \, d z d t
   - \iinte_{\Pi_T} h_l \, u_t \, \eta_z \, d z d t \nonumber\\[-1.5ex]
   \label{17}\\[-1.5ex]
   &= \iinte_{\Pi_T} r \, j_l \, \eta_z \, d z d t + \int_{\T} h_{l0} \, \eta(0) \, d z,  \quad (l=1,2),\nonumber
\end{align}
\begin{align}
   &- \iinte_{\Pi_T} u_t \, \zeta_t \, dz d t + \iinte_{\Pi_T} \nu^2 \, u_z \, \zeta_z \, d zd t
   + \iinte_{\Pi_T} p \, (\hh^2)_z \, \zeta \, d z d t\nonumber\\[-1.5ex]
   \label{18}\\[-1.5ex]
   &= \iinte_{\Pi_T}f\zeta \, d z d t + \int_{\T} u_1 \, \zeta(0) \,d z, \quad u(0)= u_0,\nonumber
\end{align}
for all $\eta,\zeta \in {W}^{1,1}_{2}(\Pi_T)$, which are equal to zero for $t=T$.
\end{definition}

Moreover, it is also possible to define the weak solution of \eqref{9}--\eqref{16} in a way somewhat differently.

\begin{definition}
\label{def2} A pair of 1-periodic functions in the $z$-spacial variable
$$
  \Big( \hh \in V_{2}(\Pi_T), \quad u \in W^{1,1}_{2}(\Pi_T) \Big)
$$
is called a weak solution of the initial-value problem \eqref{9}--\eqref{16} if satisfies, for almost all $t_1 \in [0,T]$, the identities
\begin{align}
   &-\int_0^{t_1} \! \! \! \int_\T  h_l \, \eta_t \, dz dt + \int_0^{t_1} \! \! \! \int_\T r h_{lz} \, \eta_z \, dz dt
   - \int_0^{t_1} \! \! \! \int_\T h_l \, u_t \, \eta_z \, dz dt \nonumber\\[-1.5ex]
   \label{19}\\[-1.5ex]
   &= \int_0^{t_1} \! \! \! \int_\T r\, j_l \, \eta_z \, d z d t + \int_{\T}h_{l0} \, \eta(0)\, dz
   - \int_{\T} h_l(t_1) \, \eta(t_1) \, d z, \,(l=1,2),\nonumber
\end{align}
\begin{align}
   &- \int_0^{t_1} \! \! \! \int_\T u_t \, \zeta_t \, dz dt + \int_0^{t_1} \! \! \! \int_\T \nu^2 \, u_z \, \zeta_z \, dz dt + \int_0^{t_1} \! \! \! \int_\T p \, (\hh^2)_z \, \zeta \, dz \, dt \nonumber\\[-1.5ex]
   \label{20}\\[-1.5ex]
   &= \int_0^{t_1} \! \! \! \int_\T f \, \zeta \, dz dt + \int_{\T} u_1 \, \zeta(0) \, dz - \int_{\T}u_t(t_1) \,  \zeta(t_1) \, dz, \quad u(0)= u_0, \nonumber
\end{align}
for any $\eta, \zeta \in {W}^{1,1}_2({\Pi_T})$.
\end{definition}
Both definitions are equivalent. The fulfillment of transmission conditions \eqref{16} is understand in the sense of the identities considered in Definition \ref{def1}.

\bigskip
%%%%%%%%%%%%%%%%%%%%%%%%%%%%%%%%%%%%%%%%%%%%%%%%%%%%%%%%%%%%%%%%%%%%%%%%%%%%%
\section{Well-posedness}\label{s.3}
%%%%%%%%%%%%%%%%%%%%%%%%%%%%%%%%%%%%%%%%%%%%%%%%%%%%%%%%%%%%%%%%%%%%%%%%%%%%%

The main focus of this section is to establish in the sense of Hadamard the well-posedness of the periodically Cauchy diffraction problem \eqref{9}--\eqref{16}, that is, we prove existence of a solution, show the uniqueness and the continuous dependence in all the parameters and data.

%%%%%%%%%%%%%%%%%%%%%%%%%%%%%%%%%%%%%%%%%%%%%%%%%%%%%%%%%%%%%%%%%%%%%%%%%%%%%
\subsection{Existence of weak solution}\label{EWS}
%%%%%%%%%%%%%%%%%%%%%%%%%%%%%%%%%%%%%%%%%%%%%%%%%%%%%%%%%%%%%%%%%%%%%%%%%%%%%

In this section we show a week solution of the periodically Cauchy diffraction problem \eqref{9}--\eqref{16}.
\begin{theorem}$(${\bf Existence theorem}$)$.
\label{TIVP} Let $\hh_0 \in C^\alpha(\T)$, $\big(\alpha \in
(0,1)\big)$ and $(u_0, u_1) \in W_2(\T) \times L_2(\T)$ be given
initial-data to the Cauchy problem \eqref{9}--\eqref{16}.
Moreover, assume that the constant $p$, the functions $r, \nu$ and
the free members $\jj,f,$ satisfy the properties:
\begin{list}{}{\parsep=0pt \itemsep=1ex \topsep=1ex}
\item[(a)] $r, \nu,\bfj,f$ are 1-periodic functions, piecewise smooth, bounded and discontinuous at each point $z= z_k + \kappa$,
$(k= 1,2,\dots,m; \; m \in \nat)$, $(\kappa \in \Z)$, with
$$
  0 < z_1 < z_2 < \ldots < z_m < 1;
$$
\item[(b)] $p$ is a positive number and there exist $0<r_0, \nu_0, r_1, \nu_1 < \infty$, such that
$$
  r_0 \leq r \leq r_1 \quad \text{and} \quad \nu_0 \leq \nu\leq \nu_1.
$$
\end{list}
Then, the Cauchy problem \eqref{9}--\eqref{16} has a weak solution
$$
   \Big(\hh \in V_2(\Pi_T),\quad u \in W^{1,1}_2(\Pi_T) \Big).
$$
\end{theorem}
\begin{proof}
1. First, following the strategy given at Remark \ref{RIBVP}, we consider the system given by \eqref{9}--\eqref{10} in $\clg{D}'(Q_T)$, the initial-boundary value data given respectively by \eqref{11}--\eqref{12} and \eqref{BC1}--\eqref{BC2} and finally the jump conditions
\eqref{15}--\eqref{16}. To solve this initial-boundary value diffraction problem we make use of Faedo-Galerkin's method.

\medskip
2. Now, we consider in ${W}^1_2(\Omega)$ a fundamental system of functions $\psi_k(z)$, $(k=1,2,\ldots)$, satisfying for each $k\geq 1$
\begin{equation}
\label{GAC}
   \psi_k(0)= \psi_k(1), \quad  \psi_{kz}(0)= \psi_{kz}(1).
\end{equation}
Moreover, we assume that $\{\psi_k\}_{k=1}^\infty$ is an orthonormal basis of $L_2(\Omega)$ and an orthogonal basis of $W^1_2(\Omega)$. Fixed a positive integer $N$, we will look for an approximating solution in the form
\begin{equation}\label{21}
    \hh^N(t):= \sum_{k=1}^N \bfa_k^N(t) \, \psi_k, \quad u^N(t):= \sum_{k=1}^Nb^N_k(t) \, \psi_k,
\end{equation}
where for $0 \leq t \leq T$. The coefficients $\bfa^N_k\equiv(a^N_{1k},a^N_{2k})$,
$$
  a^N_{lk}(t)= (h^N_l(t),\psi_k), \quad  b^N_k(t)= (u^N(t),\psi_k)\,,l=1,2\,,k=1, \ldots, N\,,
$$
are determined from the equations
\begin{align}
  \big(h^N_l,\psi_k \big)_t & = \big( -r h^N_{lz} +  h^N_l\, u_t^N + rj_l,\psi_{kz} \big)\,,\nonumber\\[-1.5ex]
  \label{22}\\[-1.5ex]
  (h^N_l(0),\psi_k) &= h_{l0k},\,l=1,2\,,k=1, \ldots, N\,,\nonumber
\end{align}
\begin{align}
 ( u^N,\psi_k)_{tt} &= (-\nu^2 u_z^N + p\hh^2,\psi_{kz}) + (f,\psi_k),\nonumber\\[-1.5ex]
  \label{23}\\[-1.5ex]
  (u^N(0),\psi_k) &= u_{0k},\quad (u^N(0),\psi_k)_t = u_{1k}\,,k=1, \ldots, N\,.\nonumber
\end{align}
Moreover, the values $\hh_{0k}\equiv(h_{10k},h_{20k}),u_{0k},u_{1k}$ are the Fourier coefficients in $L_2(\Omega)$ of initial data with respect to the system of the functions $\psi_k$.

\medskip
3. The equations (\ref{22}) and (\ref{23}) form a system of nonlinear ordinary differential equations, and from standard  existence theory, there exists a unique local solution on a maximal interval $[0,\tau)$, for some $\tau>0$. In fact, we will prove that $|\bfa^N_{k}|,|b^N_k|, k=1,\dots,N,$ are bounded functions for $t\in [0,T]$. Therefore, the system (\ref{22})--(\ref{23}) has a unique solution on $[0,T]$ for any
$T>0$. Indeed, we multiply the differential equations in \eqref{22} and \eqref{23} respectively by $p \, a^N_{lk}$ and $b^N_{kt}$. Now, we
sum the obtained equalities over all $k$ from 1 to $N$ and integrate the result with respect to $t$ from 0 to $t_1$. Therefore, we obtain
\begin{multline}\label{24}
    \frac{1}{2}\big(p\|\bfh^N(t)\|^2_{2,\Omega} + \|u^N_t(t)\|^2_{2,\Omega} + \|\nu u^N_z(t)\|^2_{2,\Omega} \big) \Big|^{t=t_1}_{t=0}
   + p \|\sqrt{r}\bfh^N_{z}\|^2_{2,Q_{t_1}}\\
    = p \iinte_{Q_{t_1}}r \bfj \cdot \bfh^N_{z} \, dz dt + \iinte_{Q_{t_1}}fu^N_t \, dz dt,
\end{multline}
where $Q_{t_1}= (0,t_1) \times \Omega$. One observes that
$$
\begin{aligned}
   &\|\bfh^N(0)\|^2_{2,\Omega} = \sum^N_{k=1}\bfa^2_{k}(0)\leq\|\bfh_{0}\|^2_{2,\Omega},
   \\
   &\|u_t^N(0)\|^2_{2,\Omega} = \sum^N_{k=1}b^2_{kt}(0)\leq\|u_1\|^2_{2,\Omega},
   \\
   &\|\nu \, u_z^N(0)\|^2_{2,\Omega}\leq \mu_0 \, \nu_1^2\|u_{0z}\|^2_{2,\Omega},
   \\[5pt]
   &- \frac{1}{2}\iinte_{Q_{t_1}} r \big((\bfh^N_{z})^2-2\bfh^N_{z}\cdot\bfj+\bfj^2\big) \, dz dt \leq 0,
\end{aligned}
$$
where the positive constant $\mu_0$ does not depend on $N$. Consequently, it follows that
\begin{multline*}
   \frac{p}{2}\|\bfh^N(t_1)\|^2_{2,\Omega}  + \frac{1}{2}\|u^N_t(t_1)\|^2_{2,\Omega}
   + \frac{1}{2}\|\nu u^N_z(t_1)\|^2_{2,\Omega} + \frac{p}{2} \|\sqrt{r} \, \hh_{z}^N\|^2_{2,Q_{t_1}}\\
   \leq \mu_1 + \iinte_{Q_{t_1}}|fu^N_t| \, dz d t,
\end{multline*}
where the positive constant
$$
   \mu_1 = \frac{p}{2}\|\bfh_{0}\|^2_{2,\Omega}
   + \frac{1}{2}\|u_1\|^2_{2,\Omega} + \frac{\mu_0\nu_1^2}{2}\|u_{0z}\|^2_{2,\Omega}
    + \frac{p}{2} \iinte_{Q_{t_1}}r\bfj^2 \, dz dt
$$
does not depend on $N$. In particular, we have
\begin{equation}\label{25}
   \frac{1}{2}\|u_t^N(t_1)\|^2_{2,\Omega}\leq \mu_1 + |\iinte_{Q_{t_1}}fu^N_t \, dz dt|
\end{equation}
and integrating this inequality with respect to $t_1$ from $0$ to $T$, we obtain
\begin{equation}\label{26}
   \frac{1}{2}\|u_t^N\|^2_{2,Q_T}\leq \mu_1 T + \int_0^T|\iinte_{Q_{t_1}} f u^N_t \, dz dt| \, dt_1\,.
\end{equation}
Applying the generalized Young's inequality, we have
\begin{equation}\label{27}
   |\iinte_{Q_{t_1}}fu^N_t \, dz dt|\leq\frac{\epsilon}{2}\iinte_{Q_T}(u^N_t)^2 \, d z dt
   + \frac{1}{2\epsilon} \iinte_{Q_T} f^2 \, dz dt
\end{equation}
and setting $\epsilon=1/2T$, we obtain
\begin{equation}\label{28}
   \int_0^T| \iinte_{Q_{t_1}} fu^N_t \, dz dt| \, dt_1 \leq\frac{1}{4} \iinte_{Q_T}(u^N_t)^2 \, dz dt + T^2 \iinte_{Q_T}f^2 \, dz dt\,.
\end{equation}
Therefore, inequalities (\ref{26}) and (\ref{28}) yield
$$
   \frac{1}{4}\|u^N_t\|^2_{2,Q_T}\leq\mu_1T + T^2 \iinte_{Q_T}f^2 \, dz dt \,.
$$
From the latter inequality  and \eqref{27} with $\epsilon=1/2T$,
it follows that
$$
\begin{aligned}
|\iinte_{Q_{t_1}} fu^N_t \, dz dt | &\leq \frac{1}{4T}\iinte_{Q_T}(u^N_t)^2 \, dz d t + T \iinte_{Q_T} f^2 \, dz dt \\
&\leq\mu_1 + 2T \iinte_{Q_T} f^2 \, dz dt\,.
\end{aligned}
$$
Consequently, from (\ref{24}) we have for any $t_1 \in(0,T]$
\begin{multline}\label{29}
   p \, \|\bfh^N(t_1)\|^2_{2,\Omega} + \frac{1}{2}\|u^N_t(t_1)\|^2_{2,\Omega}
   + \frac{1}{2}\|\nu u^N_z(t_1)\|^2_{2,\Omega} + p \, \|\sqrt{r}\bfh_{z}^N\|^2_{2,Q_{t_1}}
   \\[5pt]
   \leq 2 \mu_1 + 2T \iinte_{Q_T} f^2 \, dz dt =:\mu_2,
\end{multline}
where the positive constant $\mu_2$ does not depend on $N$. It follows from (\ref{29}) that $\bfa^N_{k},b^N_k,b^N_{kt}$ are uniformly bounded functions for each $t \in [0,t_1]$, $t_1\leq T$.

\begin{remark} \label{RET}
One observes that, for each $t \in [0,t_1]$, $t_1\leq T$
$$
  \int_\Omega \big( \bfh^N(t,z) - \bar{\bfh}^N(t)  \Big) \, dz= 0,
$$
where $\bar{\bfh}^N$ is the average of $\bfh^N$ on $\Omega$, that is
$$
  \bar{\bfh}^N(t):= \medint_{\!\!\!\Omega} \bfh^N(t,z) \, dz \leq \|\bfh^N(t)\|_{2,\Omega}
  \leq \Big(\frac{\mu_2}{p}\Big)^{1/2}.
$$
\end{remark}

\bigskip
4. Now, let us show that for an arbitrary fixed $k \leq N$, the $\bfa^N_{k},b^N_k,b^N_{kt}$ are also equicontinuous on $[0,T]$. For any $\delta>0$, we set $Q_{t,t+\delta}= (t,t+\delta) \times \Omega$. Then, we have from (\ref{22})
$$
  \begin{aligned}
   |\bfa^N_{k}(t + \delta) - \bfa^N_{k}(t)| &= \left|\iinte_{Q_{t,t+\delta}}(-r\bfh^N_{z} + u^N_t\bfh^N + r \bfj)\, \psi_{kz} \, dz dt \right|
   \\[5pt]
   &\leq \iinte_{Q_{t,t+\delta}} |r \, \bfh^N_{z} \, \psi_{kz}| \, dzdt + \iinte_{Q_{t,t+\delta}} |u^N_t \, \bfh^N \, \psi_{kz}| \, dzdt \\[5pt]
   &+ \iinte_{Q_{t,t+\delta}} |r \, \bfj \, \psi_{kz}| \,
   dzdt =: I_1 + I_2 + I_3,
  \end{aligned}
$$
with obvious notation. For convenience, we recall generalized Hölder's inequality
\begin{equation}
   \label{30}
   \iinte_{Q_T}|v_1 \, v_2 \, v_3| \, dz dt\leq \prod_{j=1}^3 \; \|v_j\|_{q_j,\gamma_j,Q_T},
\end{equation}
where
$$
q_i,\gamma_i\in[1,\infty),\,i=1,2,3,\quad \frac{1}{q_1}+\frac{1}{q_2}+\frac{1}{q_3}=1,\quad\frac{1}{\gamma_1}+\frac{1}{\gamma_2} +\frac{1}{\gamma_3}=1\,.
$$
Then, applying the above inequality, we could write
$$
\begin{aligned}
   I_1 &\leq r_1 \|\bfh^N_z\|_{2,Q_{t,t+\delta}} \, \|\psi_{kz}\|_{2,Q_{t,t+\delta}}
\\[5pt]
   &\leq  r_1 \, \Big(\frac{\mu_2}{p \, r_0}\Big)^{1/2} \, \delta^{1/2} \; \|\psi_{kz}\|_{2,\Omega},
\\[5pt]
   I_2 &\leq \|u^N_t\|_{2,4,Q_{t,t+\delta}}\; \|\bfh^N\|_{\infty,4,Q_{t,t+\delta}} \;\|\psi_{kz}\|_{2,Q_{t,t+\delta}}
   \\[5pt]
   &\leq \mu_2 \; \delta^{1/4} \;\|\bfh^N\|_{\infty,4,Q_{t,t+\delta}} \, \delta^{1/2} \;\|\psi_{kz}\|_{2,\Omega}
   \\[5pt]
   &\leq \mu_2 \; \delta^{3/4} \; \|\bfh^N\|_{\infty,4,Q_{t,t+\delta}}\,\|\psi_{kz}\|_{2,\Omega}
   \\[5pt]
   I_3 &\leq r_1 \| \jj \|_{2,Q_{t,t+\delta}} \,\|\psi_{kz}\|_{2,Q_{t,t+\delta}}
   \\[5pt]
   &\leq r_1 \, \mu_2 \, \delta^{1/2}\|\psi_{kz}\|_{2,\Omega}.
\end{aligned}
$$
%where we have used Lemma \ref{} to estimate $I_2$.
%
In order to estimate $\|\bfh^N\|_{\infty,4,Q_{t,t+\delta}}$, first we recall Remark \ref{RET} and define
$$
  \bfv^N:= \bfh^N - \bar{\bfh}^N.
$$
Then, we apply Theorem 2.2. in \cite[pp.62-63]{LSU}, with Remark 2.1 in \cite[pp.63]{LSU} for each component of the vector-function $\bfv^N$. Therefore, we have
$$
\|v\|_{q,\Omega}\leq\beta\|v_z\|^{\frac{2}{\gamma}}_{2,\Omega}\cdot\|v\|^{1-\frac{2}{\gamma}}_{2,\Omega},
$$
where
$$
 q\in[2,+\infty],\quad\gamma\in[4,+\infty],\quad\frac{1}{\gamma}+\frac{1}{2q}=\frac{1}{4}\,,
$$
the positive constant $\beta$ depends on $\gamma$ and $\Omega$, and $v$ stands for the components of $\bfv^N$. By integration with respect to the time variable from $t$ to $t +\delta$, we obtain
$$
  \|v\|_{q,\gamma,Q_{t,t+\delta}}\leq\beta\|v_z\|^{\frac{2}{\gamma}}_{2,Q_{t,t+\delta}}\;
  \mbox{vrai} \max_{\tau\in[t,t+\delta]}\|v\|^{1-\frac{2}{\gamma}}_{2,\Omega}.
$$
Moreover, using Young's inequality, we could rewrite the latter inequality as
\begin{equation}
\label{33}
   \|v\|_{q,\gamma,Q_{t,t+\delta}}\leq\beta \frac{2}{\gamma}\|v_z\|_{2,Q_{t,t+\delta}}
   + \beta(1-\frac{2}{\gamma})\; \mbox{vrai}\max_{\tau\in[t,t+\delta]}\|v\|_{2,\Omega}\,.
\end{equation}
Consequently, from equation \eqref{33}, it follows that
$$
  \begin{aligned}
  I_2 &\leq \mu_2 \, \delta^{3/4} \, \beta\, \Big( \|\bfh^N_z\|^{1/2}_{2,Q{t,t+\delta}} +
  \mbox{vrai}\max_{\tau\in[t,t+\delta]}\|\bfh^N(\tau)\|^{1/2}_{2,\Omega}\Big) \, \|\psi_{kz}\|_{2,\Omega}\\
   &+ \mu_2 \, \delta^{3/4} \, \beta \, \Big( \|\bar{\bfh}^N_z\|^{1/2}_{2,Q{t,t+\delta}}+
  \mbox{vrai}\max_{\tau\in[t,t+\delta]}\|\bar{\bfh}^N(\tau)\|^{1/2}_{2,\Omega}\Big) \, \|\psi_{kz}\|_{2,\Omega}\\
  &\leq 2 \, \mu_2 \, \delta^{3/4} \, \beta\, \big( \Big( \frac{\mu_2}{p \, r_0} \Big)^{1/4} + \Big( \frac{\mu_2}{p} \Big)^{1/4}
  \, \big)\|\psi_{kz}\|_{2,\Omega}.
  \end{aligned}
$$

\bigskip
Therefore, given $\epsilon>0$, there exists a $\delta>0$ such that, if $0 < t \leq \tau < t+\delta$, then
$$
|\bfa^N_{k}(\tau)-\bfa^N_{k}(t)|\leq \epsilon \;\|\psi_{kz}\|_{2,\Omega},
$$
where $\epsilon$ does not depend on $N$ and tends to zero as $\delta \to 0^+$, i.e. the uniformly equicontinuity in $t$ of the $\bfa^N_{k}$. The equicontinuity of functions $b^N_k$ follows from the boundedness of their derivatives. Moreover, it is proved similarly as done for the functions $\bfa^N_{k}$, $N\geq k$, that $b^N_{kt}$ are equicontinuous functions on $t$ for $N\geq k$.

\bigskip
5. By usual Cantor's diagonal process, we can select subsequences
$$
  \{\bfa_{k}^{N_m}\}_{m=1}^\infty \subset \{\bfa_{k}^{N}\}_{N=1}^\infty \quad \text{and} \quad
  \{b_k^{N_m}\}_{m=1}^\infty \subset \{b_k^{N}\}_{N=1}^\infty
$$
be converging uniformly on $[0,T]$ to some continuous functions $\bfa_{k}(t),b_k(t)$. Hence for each $t \in [0,T]$, we define
\begin{equation}
\label{LF}
   \hh(t) := \sum_{k=1}^{\infty}\bfa_{k}(t) \, \psi_k,\qquad u(t) := \sum_{k=1}^{\infty}b_k(t) \, \psi_k\,.
\end{equation}
The sequence $\{\bfh^{N_m}\}$ converges to the function $\bfh$ weakly in $L_2(\Omega)$ and uniformly with respect to $t$ in $[0,T]$. Indeed, for any $\psi \in L_2(\Omega)$, taking $\ppsi= \psi \bfe$, with $\bfe=(1,1)$, it follows that
\begin{multline}\label{34}
   \big(\hh^{N_m}-\hh,\ppsi \big)\\
   =\sum_{k=1}^s(\psi,\psi_k)(\hh^{N_m}-\hh,\psi_k\bfe) +(\sum_{k=s+1}^{\infty}(\hh^{N_m}-\hh,(\psi,\psi_k)\psi_k\bfe),
\end{multline}
with
$$
|(\bfh^{N_m}-\bfh,\sum_{k=s+1}^{\infty}(\psi,\psi_k)\psi_k\bfe)|\leq C_1(\sum_{k=s+1}^{\infty}(\psi,\psi_k)^2)^{1/2}\equiv C_1R(s),
$$
where the positive constant $C_1$ does not depend on $N_m$ and $s$. Now, we choose $s$ large enough so that, $C_1R(s)$ becomes less than a preassigned $2\epsilon>0$. On the other hand, for fixed $s$ and large enough $N_m$, the first sum in (\ref{34}) will be less than $\epsilon$ for all $t$ in $[0,T]$. Therefore, the term
$$
  |\big( \hh^{N_m}-\hh,\ppsi \big)|
$$
can be made less than $\epsilon$ for all $t$ in $[0,T]$. Consequently, the sequence $\{\bfh^{N_m}\}$ converge to $\bfh$ weakly in $L_2(\Omega)$, uniformly with respect to $t\in[0,T]$. Moreover, by construction and uniqueness of the limit, we obtain for each $t \in [0,T]$
\begin{equation}
\label{BCH}
   \hh(t,0)= \hh(t,1) \qquad \hh_z(t,0)= \hh_z(t,1).
\end{equation}
The sequence $\{u^{N_m}\}$ is bounded in $L_{\infty}(0,T;W^1_2(\Omega))$ and the sequence $\{u^{N_m}_t\}$ is bounded in $L_{\infty}(0,T;L_2(\Omega))$. Hence
$$
\text{$u^{N_m}$ converges to $u$ $*$-weakly in $L_{\infty}(0,T;W^1_2(\Omega))$},
$$
$$
 \text{$u^{N_m}_t$ converges to $u_t$ $*$-weakly in $L_{\infty}(0,T;L_2(\Omega))$}.
$$
Moreover, the sequence $\{u^{N_m}\}$ is bounded in $W^{1,1}_2(Q_T)$. Then, from the Sobolev's Imbedding Theorem the sequence $\{u^{N_m}\}$ converges strongly to $u$ in $L_2(Q_T)$ and, passing to an appropriate subsequence if necessary, we obtain a.e. convergence in $Q_T$. Finally, we observe that for a.e. $t\in [0,T]$
\begin{equation}
\label{BCU}
   u(t,0)= u(t,1) \qquad u_z(t,0)= u_z(t,1).
\end{equation}

\bigskip
6. At this point, we proceed to extend by periodicity the considered functions,  where conditions \eqref{BCH} and \eqref{BCU} will be used. For convenience, we remain the same notation. Therefore, we have
\begin{equation}
\text{$\bfh^{N_m}(t)$ converges to the function $\bfh(t)$ weakly in $L_2(\T)$},
\end{equation}
uniformly with respect to $t$ in $[0,T]$ and
\begin{equation}
  \text{$u^{N_m}$ converges to $u$ $*$-weakly in $L_{\infty}(0,T;W^1_2(\T))$},
\end{equation}
\begin{equation}
  \text{$u^{N_m}_t$ converges to $u_t$ $*$-weakly in $L_{\infty}(0,T;L_2(\T))$}.
\end{equation}
Now, let us show that the sequence $\{u^{N}_{tt}\}$ is bounded in $L_2(0,T;H^{-1}(\T))$, where $H^{-1}(\T)$ is the dual space of
$H^1(\T)\equiv W^{1}_2(\T)$, see Iório--Iório \cite{IO}. For this purpose, we fix any function $\Psi \in W^{1}_2(\T)$, such that
$\|\Psi\|_{W^{1}_2(\T)} \leq 1$, $\Psi=\Psi_1 + \Psi_2$, where
$$
  \text{$\Psi_1\in \mbox{span}\{\psi_k\}_{k=1}^{k=N}$, and $(\Psi_2,\psi_k)=0, k=1,\dots,N$}.
$$
We denote by  $\langle a,b \rangle$ the pairing between $a \in H^{-1}(\T)$ and $b \in W^{1}_2(\T)$.  From (\ref{23}), we have
$$
\langle u_{tt}^N, \Psi\rangle = (u_{tt}^N, \Psi_1) = -(\nu^2u^N_{z},\Psi_{1z}) + p((\hh^N)^2,\Psi_{1z}) + (f,\Psi_{1})\,.
$$
It is easy to check that $\|\Psi_1\|_{W^{1}_2(\T)}\leq1$. In view of this and estimates obtained before, we obtain
$$
|\langle u^N_{tt},\Psi\rangle|\leq C_2\mu_2,
$$
where the positive constant $C_2$ does not depend on $N$. It proves that $\{u^N_{tt}\}$ is bounded in $L_2(0,T;H^{-1}(\T))$. Then, it follows that
\begin{equation}
  u^N_{tt}\rightharpoonup u_{tt}\;\mbox{weakly in}\;L_2(0,T;H^{-1}(\T))\,.
\end{equation}
Analogously, let us show that the sequence $\{\bfh^N_{t}\}$ is bounded in $L_2(0,T;H^{-1}(\T))$. First from (\ref{29}), one can extract from the sequence $\{\bfh^{N_m}\}$ a subsequence converging to $\bfh$ weakly in $L_2(\Pi_T)$ together with $\{\bfh^{N_m}_{z}\}$. Now consider a function $\Phi\in W^{1}_2(\T)$, such that $\|\Phi\|_{W^{1}_2(\T)} \leq 1$, $\Phi=\Phi_1 + \Phi_2$, where $\Phi_1\in \mbox{span}\{\psi_k\}_{k=1}^{k=N}$, and $(\Phi_2,\psi_k)=0, k=1,\dots,N$. Since we have
$$
h_{lt}^N=\sum_{k=1}^Na^N_{lkt}\psi_k\in\mbox{span}\{\psi_k\}_{k=1}^{k=N},\quad l=1,2,
$$
hence the following equalities are valid
$$
\langle h^N_{lt},\Phi\rangle = \langle h^N_{lt},\Phi_1\rangle = -(rh^N_{lz}-u^N_th^N_l-rj_l,\Phi_{1z}),\,l=1,2.
$$
Using $\|\Phi_1\|_{W^1_2(\Omega)}\leq1$ and \eqref{29}--\eqref{33}, we obtain
$$
  |\langle h^N_{lt},\Phi\rangle|\leq C_3 \mu_2, \quad l=1,2,
$$
where the positive constant $C_3$ does not depend on $N$. The latter provides that $\{\bfh^N_t\}$ is bounded in $L_2(0,T;H^{-1}(\T))$. Also from \eqref{29}, it follows that $\{\bfh^N\}$ is bounded in $L_2(0,T;W^1_2(\T))$. The well-known Aubin-Lions Theorem, \cite[p.58]{Li}, implies that, a subsequence of $\{\bfh^N\}$ converges strongly in $L_2(\Pi_T)$. Therefore, passing a subsequence if necessary, we have a.e. convergence in
$\Pi_T$. Without loss of generality, we can assume that the sequences $\{\bfh^{N_m}\},\{u^{N_m}\}$ converge to the limit functions
$\bfh,u$ in the sense mentioned above. Consequently, the sequence $\{\bfh^N\}$ converges a.e. in $\Pi_T$.

\bigskip
7. Finally, it remains to show that the pair $(\hh, u)$ of limit functions satisfy equations \eqref{17} and \eqref{18}. First, we will show that the function $\hh$ satisfies equality \eqref{17}. For this purpose we multiply each equation of (\ref{22}) by a smooth function $\alpha_k(t)$ that is equal to zero for $t=T$, then sum over all $k$ from 1 to $N'\leq N$, and integrate the result with respect to $t$ from $0$ to $T$. Therefore, we obtain
\begin{multline}
\label{35}
   \int_0^T(h^N_l,\Upsilon^{N'}_t) \, dt = \int_0^T\big((rh^N_{lz},\Upsilon^{N'}_z) - (u^N_th^N_l,\Upsilon^{N'}_z)-(rj_l,\Upsilon^{N'}_z)\big) \, dt\\
   + (h^{N}_{l0},\Upsilon^{N'}_z(0)),\,l=1,2,
\end{multline}
where $\Upsilon^{N'}(t,z)=\sum_{k=1}^{N'}\alpha_k(t)\psi_k(z)$ belongs to $L_{\infty}(0,T;H^1_0(\T))\hookrightarrow L_{\infty}(\Pi_T)$. We claim that we can pass to the limit in equation (\ref{35}) along the subsequences selected above, assuming $\Upsilon^{N'}$ fixed, and thereby arrive at (\ref{35}) with $h_l^{N_m}, u^{N_m}$ being replaced respectively by $h_l,u$, i.e. we must prove that
$$
   \int_0^T (u_t^{N_m}h_l^{N_m}-u_th_l,\Upsilon^{N'}_z) \, dt \rightarrow 0,\quad\mbox{as}
   \quad m\rightarrow\infty,\quad  \mbox{for} \quad l=1,2.
$$
Indeed, by a simple algebraic manipulation, we have
$$
  \begin{aligned}
   \int_0^T (u_t^{N_m}h_l^{N_m}-u_th_l,\Upsilon^{N'}_z) \, dt &= \int_0^T(u_t^{N_m}(h_l^{N_m}-h_l),\Upsilon^{N'}_z)\,dt
   \\
   &+\int_0^T(u_t^{N_m}-u_t,h_l\Upsilon^{N'}_z) \, dt \,.
  \end{aligned}
$$
But the first term in the right hand side of the latter equality tends to zero as $m\rightarrow\infty$ according to the estimates
$$
\begin{aligned}
   \int_0^T|u_t^{N_m}(t)&(h_l^{N_m}-h_l)(t),\Upsilon_z^{N'}(t)| \,
   dt \\
   &\leq \|\Upsilon_z^{N'}\|_{L_{\infty}(\Pi_T)}\int_0^T\|u_t^{N_m}(t)\|_{L_2(\T)}\|(h_l^{N_m}-h_l)(t)\|_{L_2(\T)} \, dt \\
   &\leq 2 \mu_2 T \|\Upsilon_z^{N'}\|_{L_{\infty}(\Pi_T)}\|(h_l^{N_m}-h_l)\|_{L_{\infty}(0,T;L_2(\T))}\rightarrow 0,
\end{aligned}
$$
as $m \to \infty$. Moreover, one observes that the last term in \eqref{35} tends to 0, since we have proved that $u_t^{N_m}$
converges to $u_t$ $\ast$-weakly in $L_{\infty}(0,T;L_2(\T))$ and $h_l\Upsilon_z^{N'}\in L_2(0,T;L_2(\T))\hookrightarrow
L_1(0,T;L_2(\T))$ due to the estimates
$$
   \|h_l\Upsilon_z^{N'}\|_{L_2(0,T;L_2(\T))}\leq\|h_l\|_{L_2(0,T;L_{\infty}(\T))}\|\Upsilon_z^{N'}\|_{L_{\infty}(0,T;L_2(\T))},\quad l=1,2.
$$
But it is not difficult to show that $\Upsilon^{N'}$ are dense in the space of all required functions in the first definition of a weak solution. In view of this $h_l, l=1,2,$ satisfy equalities \eqref{17} and belong to be the space $V_2(\Pi_T)$. Moreover, since
$$
  \|u_t^N\|_{2,\Omega}^2 \leq 2 \, \mu_2,
$$
passing to the limit, we obtain
$$
  \|u_t\|_{2,\Omega}^2 \leq 2 \, \mu_2.
$$
Therefore, we have
$$
   u^2_t\in L_{1,\frac{2}{1-2\kappa}}(\Pi_T),\quad\kappa\in(0,\frac{1}{2}),
$$
and we can deduce from Theorem 8.1 \cite[p.192]{LSU}, interpreting $Q= [0,T] \times [-1,2]$, $Q'= [0,T] \times [0,1]$ and by the periodicity procedure, that
$$
  \mbox{vrai}\max_{\Pi_T}|\bfh|\leq C_4,
$$
for some positive constant $C_4$. Similarly, from Theorem 10.1 of \cite[p.204]{LSU}, it follows that
$$
  \bfh\in C^{\alpha_1,\,\alpha_1/2}(\Pi_T),
$$
where $0 < \alpha_1 \leq \alpha$. Now, let us show that $u$ satisfies equality (\ref{18}). For this purpose, first we multiply each equation of (\ref{23}) by the $\alpha_k(t)$, then sum over all $k$ from 1 to $N'\leq N$, and integrate the result with respect to $t$ from $0$ to $T$. Hence after an integration by parts, we have
\begin{multline}
    \label{36}
    \int_0^T(u^N_t,\Upsilon^{N'}_t) \, dt=\int_0^T\big((\nu^2u^N_z,\Upsilon^{N'}_z)- p((\hh^N)^2,\Upsilon^{N'}_z) + (f,\Upsilon^{N'})\big) \, dt
    \\
    - (u^N_{1},\Upsilon^{N'}(\cdot,0)),\;u^N(z,0)=u^N_0(z),
\end{multline}
where $\Upsilon^{N'}$ were defined in (\ref{35}). Then, passing to the limit in \eqref{36} with respect to the subsequence $\{N_m\}$
selected above, assuming that $\Upsilon^{N'}$ is fixed, we arrive at (\ref{36}) with $\hh^N,u^N$ being replaced by $\bfh,u$. Since
$\{\bfh^N\}$ converges a.e. in $\Pi_T$ to $\bfh$, we can pass to limit in the non-linear term $\{(\bfh^N)^2\}$, which converges to
$(\bfh)^2$ weakly in (\ref{36}). Moreover, since $\max_{\Pi_T}|\bfh|\leq C$, we have that
$$
  \Big\{\int_0^T (\bfh^2,\Upsilon^{N'}_z) \, dt \Big\}
$$
is bounded for any $\Upsilon^{N'}\in W^{1,1}_2(\Pi_T)$. Therefore, as the functions $\Upsilon^{N'}$ are dense in the space considered
in the definition of weak solution, we conclude that the function $u$ satisfies equality (\ref{18}) and is a weak solution from
$W^{1,1}_2(\Pi_T)$. Consequently, Theorem \ref{TIVP} is proved.
\end{proof}

\medskip
Membership of such a solution $\bfh$ in $V^{1,1/2}_2(\Pi_T)$ follows from Lemma 4.1 of \cite[p.158]{LSU} and Theorem
\ref{TIVP}. Then, we have the following

\begin{corollary}\label{cor1}
Any weak solution $\bfh$ of problem \eqref{9}--\eqref{16} from $V_2~(\Pi_T)$ belongs to $V_2^{1,1/2}(\Pi_T)$.
\end{corollary}

Moreover, we have also that
\begin{corollary}\label{cor2}
For any function $\phi \in L_2(0,T;W^1_2(\T))$ is valid the following equality
$$
   \int_0^T \langle u_{tt},\phi\rangle \, dt = \int_0^T \big((f-p (\bfh^2)_z,\phi)-(\nu^2u_z,\phi_z)\big) \, dt\,.
$$
\end{corollary}
The latter result allows us to conclude that the following equation holds for any function $\xi(z)\in W^1_2(\T)$ and for
almost all $t \in [0,T]$, that is
$$
   \langle u_{tt},\xi\rangle = (f-p(\bfh^2)_z,\xi)-(\nu^2u_z,\xi_z)\,.
$$
Furthermore, we notice that
\[
   u\in C([0,T];L_2(\T)),\;u_t\in C([0,T];H^{-1}(\T))\,.
\]

\bigskip
%%%%%%%%%%%%%%%%%%%%%%%%%%%%%%%%%%%%%%%%%%%%%%%%%%%%%%%%%%%%%%%%%%%%%%%%%%%%%%%%%%%%%%%%%%%%%%%
\subsection{Uniqueness of weak solution}\label{ss.3.2}
%%%%%%%%%%%%%%%%%%%%%%%%%%%%%%%%%%%%%%%%%%%%%%%%%%%%%%%%%%%%%%%%%%%%%%%%%%%%%%%%%%%%%%%%%%%%%%%

In this section we prove uniqueness of weak solution of the initial-value problem \eqref{9}--\eqref{16}. First, we need the following

\begin{lemma}\label{lem1}
Suppose that
$$
  \Big(\hh \in V_{2}(\Pi_T), \quad u \in W^{1,1}_{2}(\Pi_T) \Big)
$$
is a weak solution of problem \eqref{9}--\eqref{16}. Then, the following inequality is valid for almost all $t_1\in[0,T]$
\begin{multline}\label{37}
  \frac{1}{2} \int_{\T} \big(p\bfh^2(t_1) + u_t^2(t_1) + \nu^2 u_z^2(t_1)\big) \, dz
  + \frac{1}{2} \int_0^{t_1}\!\!\!\int_{\T} p \, r \; \bfh_{z}^2 \, dz dt \\[5pt]
  \leq \int_{\T}(p\bfh_{0}^2 + u_1^2 + \nu^2u_{0z}^2)\, dz + \int_0^{t_1}\!\!\!\int_{\T} p \, r \bfj^2 \, dz dt
  + 2 \, t_1 \int_0^{t_1}\!\!\!\int_{\T} f^2 \, dz dt.
\end{multline}
\end{lemma}

\begin{proof} 1. Let $\epsilon \in (0,T)$ be fixed and
consider the following functions
$$
   \hat{\eta}_{\bar{\epsilon}}(t,z) = \frac{1}{\epsilon} \int^t_{t-\epsilon}\hat{\eta}(\tau,z)\, d\tau,
   \qquad
   \hat{\zeta}_{\bar{\epsilon}}(t,z)= \frac{1}{\epsilon} \int^t_{t-\epsilon}\hat{\zeta}(\tau,z)\, d\tau,
$$
where $\hat{\eta}, \hat{\zeta} \in W^{1,1}_2((-\epsilon,T) \times\T)$, with
$$
  \text{$\hat{\eta}(\tau,z), \hat{\zeta}(\tau,z) \equiv 0$ for $\tau \in [-\epsilon,0] \cup [T-\epsilon,T]$}.
$$
Then, we take respectively $\hat{\eta}_{\bar{\epsilon}}$, $\hat{\zeta}_{\bar{\epsilon}}$ as $\eta$ and $\zeta$ in (\ref{17})--(\ref{18}). Since
$$
  (\hat{\eta}_{\bar{\epsilon}})_t = (\hat{\eta}_t)_{\bar{\epsilon}},
$$
we could rewrite the first term in (\ref{17}) in the following manner
$$
\begin{aligned}
   -\int_0^{T-\epsilon} \!\!\! \int_{\T} h_l \, \hat{\eta}_{{\bar{\epsilon}} t} \, dz dt
   &=-\int_{-\epsilon}^{T-\epsilon} \!\!\! \int_{\T} h_{l \epsilon} \, \hat{\eta}_t \, dz dt
   \\[5pt]
   &= \int_0^{T-\epsilon} \!\!\! \int_{\T} h_{l \epsilon t} \; \hat{\eta} \, dz dt,\quad l=1,2,
\end{aligned}
$$
where
\[
   h_{l \epsilon}(t):= \frac{1}{\epsilon}\int^{t+\epsilon}_t h_l(\tau) \, d\tau,\quad l=1,2.
\]
Analogously, we have
\[
   -\int_0^{T-\epsilon} \!\!\! \int_{\T} u_t \; \hat{\zeta}_{\bar{\epsilon}t} \, dz dt
   = \int_0^{T-\epsilon} \!\!\! \int_{\T} u_{\epsilon {tt}} \; \hat{\zeta} \, dz dt.
\]
where
\[
   u_{\epsilon}(t):= \frac{1}{\epsilon}\int^{t+\epsilon}_t u(\tau) \, d\tau.
\]

2. Now, we proceed similarly in the remainder terms of (\ref{17})--(\ref{18}), to transfer the averaging operation $(\cdot)_{\bar{\epsilon}}$ from $\hat{\eta}_{\bar{\epsilon}}$, $\hat{\zeta}_{\bar{\epsilon}}$ to respectively coefficients. Then, with the above notation and taking into account the permutability of the averaging operation with differentiation with respect to $z$, we obtain
\begin{align}
   &\int_0^{T-\epsilon} \!\!\! \int_{\T} \Big(h_{l \epsilon t} \, \hat{\eta}
   + \big(r h_{lz} - h_l \, u_t - r \, j_l)_\epsilon \; \hat{\eta}_z \Big)\, dz dt=0, \quad l=1,2,\nonumber\\[-1.5ex]
   \label{38}
   \\[-1.5ex]
   &\int_0^{T-\epsilon} \!\!\! \int_{\T} \Big( \nu^2 \, u_{\epsilon z} \, \hat{\zeta}_z
   + \big( u_{tt}+ p (\bfh^2)_z - f \big)_\epsilon \; \hat{\zeta} \Big) \, dz dt = 0.\nonumber
\end{align}
Let $t_1 \in [0,T-\epsilon]$ be fixed. For convenience, we choose $\hat{\zeta}$ equals zero for each $t \geq t_1$ and $\hat{\zeta}=
\zeta$ on $[0,t_1] \times \T$, for some $\zeta \in W_2^{1,1}((0,t_1) \times \T)$. Analogously, we take $\hat{\eta}=0$ for $t \geq t_1$, and by standard density argument, $\hat{\eta}= \eta$ on $[0,t_1] \times \T$, for some $\eta \in V_2^{1,0}((0,t_1) \times \T)$. Therefore, we obtain from
\eqref{38}
\begin{align*}
   &\int_0^{T-\epsilon} \!\!\! \int_{\T}\Big(h_{l \epsilon t} \, {\eta}
   + \big(r h_{lz} - h_l \, u_t - r \, j_l)_\epsilon \; {\eta}_z \Big)\, dz dt=0, \quad l=1,2,\nonumber
   \\[5pt]
   &\int_0^{T-\epsilon} \!\!\! \int_{\T} \Big( \nu^2 \, u_{\epsilon z} \, {\zeta}_z
   + \big( u_{tt}+ p (\bfh^2)_z - f \big)_\epsilon \; {\zeta} \Big) \, dz dt = 0.\nonumber
\end{align*}
Now, we are allowed to take $\eta= p \, h_{l\epsilon}$ and $\zeta=u_{\epsilon t}$ in the above equalities. Hence taking in account that
\begin{align*}
   & 2 \int_0^{t_1} \!\!\! \int_{\T} h_{l \epsilon t} \, h_{l \epsilon} \, dz dt=
   \int_{\T}h^2_{l \epsilon}(t_1,z) d z - \int_{\T}h^2_{l \epsilon}(0,z) d z,\;l=1,2,
   \\[5pt]
   &2 \int_0^{t_1} \!\!\! \int_{\T} u_{\epsilon tt} \, u_{\epsilon t} \, d z dt
    = \int_{\T} u^2_{\epsilon t}(t_1,z) d z
    - \int_{\T} u^2_{\epsilon t}(0,z) dz,
   \\[5pt]
   &2 \int_0^{t_1} \!\!\! \int_{\T} \nu^2u_{\epsilon z} \, u_{\epsilon z t} \, dz d t
   =\int_{\T} \nu^2(z) \, u^2_{\epsilon z}(t_1,z) \, dz
   -\int_{\T} \nu^2(z) \, u^2_{\epsilon z}(0,z) \, dz,
\end{align*}
and passing to the limit as $\epsilon \to 0^+$, we obtain in analogy with (\ref{24})
\begin{equation}
\label{39}
\begin{aligned}
  \frac{1}{2}\int_{\T}(p \, \bfh^2(t_1,z) &+  u_t^2(t_1,z) + \nu^2 \, u_z^2(t_1,z)) \, dz
  - \frac{1}{2}\int_{\T}(p \, \bfh_{0}^2 +  u_1^2 + \nu^2 \, u_{0,z}^2) \, d z
  \\[5pt]
  &= \int_0^{t_1} \!\!\! \int_{\T} p \, r \big( \bfj\cdot\bfh_{z} - \bfh_{z}^2 \big) \, dz dt
    + \int_0^{t_1} \!\!\! \int_{\T}  f \, u_t \, dz dt \\[5pt]
  &\leq \frac{1}{2}\int_0^{t_1} \!\!\! \int_{\T} p \, r \big( \bfj^2 - \bfh_{z}^2 \big) \, dz dt
    + \int_0^{t_1} \!\!\! \int_{\T}  |f \, u_t| \, dz dt.
\end{aligned}
\end{equation}

3. Finally, from \eqref{39} it follows that
$$
\begin{aligned}
   \int_{\T} u_t^2(t_1,z) \, dz &\leq \int_{\T}( p \, \bfh_{0}^2 + u_1^2 + \nu^2 \, u_{0z}^2) \, dz \\
   &+ \int_0^{t_1} \!\!\! \int_{\T} p \, r \, \bfj^2 \, dz dt + 2 \int_0^{t_1} \!\!\! \int_{\T} |fu_t| \, dz dt.
   \end{aligned}
$$
Then, integrating the above inequality with respect to the time variable from $0$ to $t_1$, we have
$$
\begin{aligned}
   \int_0^{t_1} \!\!\! \int_{\T} u^2_t \, dz dt
   &\leq 2 \, t_1 \int_{\T}(p \, \bfh_{0}^2 + u_1^2 + \nu^2 \, u_{0z}^2) \, dz \\
   &+ 2 \, t_1 \int_0^{t_1} \!\!\! \int_{\T} p \, r \, \bfj^2 \, dz dt + 4 \, t_1^2 \int_0^{t_1} \!\!\! \int_{\T} f^2 \, dz dt,
   \end{aligned}
$$
 where we have used Cauchy's inequality. Again by Cauchy's inequality with $\epsilon= t_1$ in (\ref{39}) and the latter result, we obtain
\begin{multline*}
  \frac{1}{2} \int_{\T} \big(p\bfh^2(t_1) + u_t^2(t_1) + \nu^2 u_z^2(t_1)\big) \, dz
  \leq \int_{\T}(p\bfh_{0}^2 + u_1^2 + \nu^2u_{0z}^2)\, dz \\[5pt]
  + \int_0^{t_1}\!\!\!\int_{\T} p \, r \big( \bfj^2 - \frac{\bfh_{z}^2}{2} \big) \, dz dt
  + 2 \, t_1 \int_0^{t_1}\!\!\!\int_{\T} f^2 \, dz dt.
\end{multline*}
\end{proof}

\medskip
Similarly we have done in the proof of the Existence theorem, item (7), we obtain from (\ref{37})
\begin{equation}\label{40}
\max_{\Pi_T}|\bfh| +  \|\bfh_{}\|_{2,\Pi_T} + \|u\|_{W^{1,1}_2(\Pi_T)} \leq C_5,\quad \bfh\in C^{\alpha_2,\,\alpha_2/2}(\Pi_T),
\end{equation}
where $C_5= C_5(T)>0$ and $\alpha_2 \in (0,\alpha]$.

\bigskip
Now, let us show the uniqueness result.
%
% about solvability of problem (\ref{9})--(\ref{16}).
%

\begin{theorem}$(${\bf Uniqueness theorem}$)$.
\label{un} Under conditions of Theorem \ref{TIVP}, the periodically Cauchy diffraction problem \eqref{9}--\eqref{16} has a unique weak solution.
\end{theorem}

\begin{proof} 1. First, let $\bfh_{k}, u_k$, $(k=1,2)$, be two weak solutions of problem (\ref{9})--(\ref{16}), well understood, for the same given initial data. For convenience, we define
\begin{equation}\label{41}
\text{$\bfv:= \bfh_{2} - \bfh_{1}$ and $w:= u_2 - u_1$}.
\end{equation}
Therefore, the functions $\bfv$ and $w$ form a weak solution of the homogeneous problem
\begin{align}
  &\bfv_{t} = \big(r \, \bfv_{z} - \bfh_2 \, w_t - u_{1t} \, \bfv \big)_z \quad \text{in $\Pi_T$}, \label{42}
  \\
  &w_{tt} = (\nu^2 \, w_z -p(\bfh_2+\bfh_1)\cdot\bfv)_z \quad \text{in $\Pi_T$},\label{43}
  \\[3pt]
%  &\bfv(\pm 1,t) =\mathbf0,\, w(\pm 1,t) = 0,\label{44}
%  \\
  &\bfv =\mathbf0,\quad w = w_t = 0 \quad \text{on $\{0\}\times\T $}.\label{45}
\end{align}
Moreover, $\bfv$ and $w$ satisfy the compatibility conditions (\ref{15}) and (\ref{16}) at the jump points. Then, we have the following equations satisfied
\begin{align}
  &\iint_{\Pi_T}\big(-v_l \, \eta_t + (rv_{lz} - h_{l2} \, w_t - u_{1t} \, v_l) \, \eta_z \big)\, dzdt = 0,\quad l=1,2,
  \label{UTE1}
  \\
  &\iint_{\Pi_T}\big(-w_t \, \zeta_t + \nu^2 \, w_z \, \zeta_z + p((\bfh_2+\bfh_1)\cdot\bfv)_z\zeta\big)\, dzdt = 0.
  \label{UTE2}
\end{align}

2. We proceed as in the proof of Lemma \ref{lem1}. Hence by analogy with (\ref{38}), we obtain from \eqref{UTE1}, \eqref{UTE2}
\begin{align}
  &\int_0^{T-\epsilon} \!\!\! \int_{\T} \big( v_{l \epsilon t}\hat{\eta}+ (r \, v_{l\epsilon} -
  h_{l2}w_t - u_{1t}v_l)_\epsilon \, \hat{\eta}_z\big) \, dzdt = 0, \quad l=1,2,\nonumber\\[-1.5ex]
  \label{46}\\[-1.5ex]
  &\int_0^{T-\epsilon} \!\!\! \int_{\T} \big(\nu^2 \, w_{\epsilon z}\hat{\zeta}_z + (w_{tt} + p((\bfh_2+\bfh_1)\cdot\bfv)_\epsilon
  \,\hat{\zeta}\big) \, dzdt = 0.\nonumber
\end{align}
Moreover, for $t_1\in(0,T-\epsilon]$, we choose conveniently $\hat{\eta}$, $\hat{\zeta}$ as
$$
 \hat{\eta}(t) = \left\{
                      \begin{array}{rl}
                        v_{l \epsilon}(t),&\, \mbox{if} \; t\in(0,t_1]\\
                        0,& \,\mbox{otherwise},
                      \end{array}\right.
$$
and
$$
\hat{\zeta}(t) = \left\{
                       \begin{array}{rl}
                           w_{\epsilon t}(t),&\, \mbox{if} \; t\in(0,t_1]\\
                           0,& \,\mbox{otherwise},
                   \end{array}\right.
$$
%where $\eta= v_{l \epsilon}$, $(l=1,2)$, $\zeta= w_{\epsilon t}$
thus we observe that
$$
\begin{aligned}
  2 \int_0^{T-\epsilon} \!\!\! \int_{\T} v_{l \epsilon t}v_{l \epsilon} \, dz dt &= \int_{\T} v_{l \epsilon}^2(t_1) \, dz
  - \int_{\T} v_{l \epsilon}^2(0) \, dz,\quad l=1,2,\\
  2 \int_0^{T-\epsilon} \!\!\! \int_{\T} w_{\epsilon tt} \, w_{\epsilon t} \, dz dt &= \int_{\T} w_{\epsilon t}^2(t_1) \, dz
  -\int_{\T} w_{\epsilon t}^2(0) \, dz,\\
  2 \int_0^{T-\epsilon} \!\!\! \int_{\T} w_{\epsilon t}w_\epsilon \, dz dt &= \int_{\T} w_\epsilon^2(t_1) \, dz
  -\int_{\T} w_\epsilon^2(0) \, dz.
\end{aligned}
$$
Now, we multiply the former equality in (\ref{46}) by $p$, sum with the second one, and with such choice and observations, it
follows that
\begin{multline}\label{47}
  \frac{p}{2}\|\bfv(t_1)\|^2_{2,\T} + \frac{1}{2}\|w_t(t_1)\|^2_{2,\T}
  + \frac{1}{2}\|\nu \, w_z(t_1)\|^2_{2,\T} + \int_0^{t_1} \!\!\! \int_{\T} p \, r \, \bfv_{z}^2 \, dz dt \\
  = \int_0^{t_1} \!\!\! \int_{\T} p \,\big(u_{1t} \, v_{1z} \, v_1 - u_{1 t} \, v_{2z} \, v_2 - h_{11z} \,
   v_1 \, w_t - h_{21z} \, v_2 \, w_t \big)\, dz dt,
\end{multline}
where we have passed to the limit as $\epsilon \rightarrow 0^+$ and used the homogeneous initial data \eqref{45}.
The result \eqref{40}, obtained from Lemma \ref{lem1}, is now used to estimate the integral in the right hand side of \eqref{47}.
Then, we have
\begin{multline*}
  \Big|\int_0^{t_1} \!\!\! \int_{\T} p \, \big(u_{1t} \, v_{1z} \, v_1 + u_{1t} \, v_{2z} \, v_2
  + h_{1z} \, v_1 \, w_t + h_{2z} \, v_2 \, w_t) \, dz dt \Big| \\
  \leq p \, C_5 \, \max_{\Pi_T}|\bfv| \Big(\|v_{1z}\|_{2,\Pi_T} + \|v_{2z}\|_{2,\Pi_T}+ \|w_t\|_{2,\Pi_T} \Big).
\end{multline*}
Therefore, for almost all $t_1 \in [0,T]$, we have from \eqref{47} and the above inequality
\begin{multline}\label{48}
  \frac{p}{2}\|\bfv(t_1)\|^2_{2,\T} + \frac{1}{2}\|w_t(t_1)\|^2_{2,\T}
  + \frac{1}{2}\|\nu \, w_z(t_1)\|^2_{2,\T} + \int_0^{t_1} \!\!\! \int_{\T} p \, r \, \bfv_{z}^2 \, dz dt \\
  \leq p \, C_5 \, \max_{\Pi_T}|\bfv| \Big(\|\bfv_{z}\|_{2,\Pi_T} + \|w_t\|_{2,\Pi_T} \Big).
\end{multline}

\medskip
3. Finally, let us show that $\bfv\equiv\mathbf{0}$ and $w\equiv0$ in $\Pi_T$. Consider the equation \eqref{42}, rewritten as
$$
  \bfv_{t} + \big(u_{1t} \, \bfv \big)_z - \big(r \, \bfv_{z} \big)_z = \big(- \bfh_2 \, w_t \big)_z,
$$
%{\bf WN: Here we should extend in time the functions, $u_1$ and $\bfh_2$. It is not clear that such functions are defined for
%times $t > T$!!! Ok., once they are, we have a linear inhomogeneous equation, which could be resolved for times $t>T$.} \\
%
and recall the homogeneous initial condition $\bfv(0)= \mathbf{0}$. Thus, if $\|\bfh_{2}w_t\|_{2,\Pi_T}=0$, then the unique solution is
$\bfv\equiv \mathbf{0}$. Consequently, by equation \eqref{43} with $w(0)=w_t(0)=0$, we have $w \equiv 0$, and we are done. On the other hand, if $\|\bfh_{2}w_t\|_{2,\Pi_T}= \varpi > 0$, then from Theorem 8.1 \cite[p.192]{LSU}, again interpreting $Q= [0,T] \times [-1,2]$, $Q'= [0,T] \times [0,1]$ and by the periodicity procedure, there exists a constant $C_6(T_0) > 0$, which depends only on $T_0$ whenever $T \in (0,T_0]$, such that
\[
\max_{\Pi_T}|\bfv / \varpi| \leq C_6(T_0),\quad \forall T\in(0,T_0].
\]
Then, it follows that
$$
  \begin{aligned}
  \max_{\Pi_T}|\bfv| &\leq  \|\bfh_{2}\|_{2,\Pi_T} \|w_t\|_{2,\Pi_T} \, C_6(T_0)
  \\
  & \leq C_5(T_0) \, C_6(T_0) \, \|w_t\|_{2,\Pi_T},
  \end{aligned}
$$
for each $T \in (0,T_0]$. Hence from \eqref{48} and the above
inequality, we obtain
$$
\begin{aligned}
   \frac{p}{2}\|\bfv(t_1)\|^2_{2,\T} &+ \frac{1}{2}\|w_t(t_1)\|^2_{2,\T} + \frac{1}{2}\|\nu w_z(t_1)\|^2_{2,\T}
   + \int_0^{t_1} \!\!\! \int_{\T} p \, r \, \bfv_{z}^2 \, dz dt
   \\[5pt]
   &\leq p \, C_5^2(T_0) \, C_6(T_0) \Big(\|\bfv_{z}\|_{2,\Pi_T}\|w_t\|_{2,\Pi_T} +  \|w_t\|^2_{2,\Pi_T} \Big)
\end{aligned}
$$
Now, applying  the Cauchy inequality with $\epsilon = 2 \, r_0 \, \big(C_5(T_0)\big)^{-2} \, \big(C_6(T_0)\big)^{-1}$, we obtain for
almost all $t_1\in[0,T]$, $T\in(0,T_0]$
$$
  \frac{p}{2}\|\bfv(t_1)\|^2_{2,\T} + \frac{1}{2}\|w_t(t_1)\|^2_{2,\T}  + \frac{1}{2}\|\nu w_z(t_1)\|^2_{2,\T}  \leq p \, C_7(T_0) \, \|w_t\|^2_{2,\Pi_T},
$$
where
$$
C_7(T_0) = C_5^2(T_0) \, C_6(T_0) \, \Big(1 + \frac{C_5^2(T_0)C_6(T_0)}{4r_0} \Big)
$$
and integrating with respect to $t_1$ from $0$ to $T$, with $T\in(0,T_0]$, we have
\begin{equation}
\label{FUR}
 \frac{p}{2}\|\bfv\|^2_{2,\Pi_T} + \frac{1}{2}\|w_t\|^2_{2,\Pi_T} + \frac{\nu_0^2}{2}\|w_z\|^2_{2,\Pi_T}\leq p \, T_0 \, C_7(T_0) \,
\|w_t\|^2_{2,\Pi_T}.
\end{equation}
The inequality \eqref{FUR} is false when
$$
  p \, T_0 \, C_7(T_0)< 1/2,
$$
whenever $w_t$, $w_z$ and $\bfv$ are non-zero functions. Consequently, taking
$$
  T_1= \min \Big\{ \Big(4 \, p \, C_7(T_0)\Big)^{-1}, T_0 \Big\},
$$
we must have $\bfv\equiv\mathbf{0},w \equiv0$ in $\Pi_{T_1}$. The
result follows applying the above procedure recursively, that is,
after a finite number of steps, we get
$\bfv\equiv\mathbf{0},w\equiv0$ in $\Pi_T$.
\end{proof}

\bigskip
%%%%%%%%%%%%%%%%%%%%%%%%%%%%%%%%%%%%%%%%%%%%%%%%%%%%%%%%%%%%%%%%%%%%%%%%%%%%%%%%%%%%%%%%%%%%%%%%%%%%%%%%%%%%%%%%%%%%%%%%%%%%%%
\subsection{Stability of weak solution}\label{SS}
%%%%%%%%%%%%%%%%%%%%%%%%%%%%%%%%%%%%%%%%%%%%%%%%%%%%%%%%%%%%%%%%%%%%%%%%%%%%%%%%%%%%%%%%%%%%%%%%%%%%%%%%%%%%%%%%%%%%%%%%%%%%%%

The aim of this section is to study the stability of the weak solutions of the periodically Cauchy diffraction problem
\eqref{9}--\eqref{16}, with respect to variations of all coefficients, the free terms and the initial-data.

\medskip
In Sections  \ref{EWS} and \ref{ss.3.2}, we have shown that the
Cauchy problem \eqref{9}--\eqref{16} has a unique weak solution
$$
  \Big(\hh \in V_{2}(\Pi_T), \quad u \in W^{1,1}_{2}(\Pi_T) \Big),
$$
satisfying some additional properties, for instance
\begin{equation}\label{49}
    \max_{\Pi_T}|\bfh| + \|\bfh_z\|_{2,\Pi_T} + \|u\|_{W^{1,1}_2(\Pi_T)} \leq C_8,
\end{equation}
where the positive constant $C_8$ does not depend on $\bfh,u$.
Therefore, for each $m \in \nat$ we consider the following
sequence of problems associated to \eqref{9}--\eqref{16}
\begin{align}
  \bfh^m_t & = \Big(r^m\bfh^m_{z} - \bfh^mu_t^m - r^m\bfj^m \Big)_z \quad \text{in $\Pi_T$},\label{50}
  \\[5pt]
  u^m_{}tt & = \Big((\nu^m)^2 \, u_z^m - p(\bfh^m)^2 \Big)_z + f^m \quad \text{in $\Pi_T$}, \label{51}
  \\[5pt]
  \bfh^m &=\bfh^m_{0} \quad \text{on $\{0\} \times \T$}, \label{52}
  \\[5pt]
  u^m &=u_0^m,\; u^m_t=u_1^m \quad \text{on $\{0\} \times \T$}. \label{53}
\end{align}

In fact, by a standard density argument, it is enough to suppose that
$$
  r^m(z), \, \bfj^m(t,z), \, \bfh_{0}^m(z),  \, u_0^m(z),\, u_1^m(z), \, f^m(t,z)
$$
are smooth functions satisfying the conditions of the uniqueness and existence theorems. Then, for each $m \in \nat$, there exists
a unique weak solution
$$
  \Big(\hh^m \in V_{2}(\Pi_T), \quad u^m \in W^{1,1}_{2}(\Pi_T) \Big),
$$
satisfying
\begin{equation}\label{49}
    \max_{\Pi_T}|\bfh^m| + \|\bfh^m_z\|_{2,\Pi_T} + \|u^m\|_{W^{1,1}_2(\Pi_T)} \leq C_9,
\end{equation}
where the positive constant $C_9$ does not depend on $\bfh^m$, $u^m$ and $m$. Moreover, the transmission conditions
\eqref{15}--\eqref{16} can be dropped owing to the smoothness of the solution.
\begin{theorem}$(${\bf Stability theorem}$)$.
\label{st} Suppose that the sequences $\{r^m\}$ and $\{\nu^m\}$ are uniformly bounded and converge a.e. to $r, \nu$  respectively.
Also the sequences $\{\bfj^m\},\{f^m\}, \{\bfh_{0}^m\},\{u_0^m\}, \{u^m_1\}$ converge to $\bfj, f$, $\bfh_{0},u_0,u_1$ in the
respectively norms of the spaces to which they belong according to conditions of Theorem \ref{EWS}.
Then, the sequence of weak solutions $\{\big(\bfh^m, u^m\big)\}_{m= 1}^\infty$,
$$
  \Big( \bfh^m \in V_2^{1,0}(\Pi_T), u^m \in W_2^{1,1}(\Pi_T)
  \Big) \qquad ( \forall \, m \geq 1),
$$
of the associated problems \eqref{50}--\eqref{53} converges in such spaces to the weak solution $(\bfh,u)$ of the limit problem \eqref{9}--\eqref{16}.
\end{theorem}
\begin{proof} We argue by contradiction. Therefore, let us assume that the sequence $\{\big(\bfh^m, u^m\big)\}_{m=1}^\infty$ does not converge to $(\bfh,u)$, and since it is uniformly bounded, we can extract a subsequence, namely $\{\big(\bfh^{m_i}, u^{m_i}\big)\}_{i= 1}^\infty$, such that
\begin{align}
  \bfh^{m_i} \rightharpoonup \hat{\bfh}   \quad &\text{weakly in $L_\infty(0,T; L_2(\T))$},\label{LMH}
  \\[5pt]
  \bfh_z^{m_i} \rightharpoonup \hat{\bfh}_z   \quad &\text{weakly in $L_2(\Pi_T)$}, \label{LMHZ}
  \\[5pt]
  u^{m_i} \rightharpoonup \hat{u}   \quad &\text{weakly in $L_\infty(0,T; W^1_2(\T))$}, \label{LMU}
  \\[5pt]
  u^{m_i}_t \rightharpoonup \hat{u}_t   \quad &\text{weakly in $L_\infty(0,T; L_2(\T))$}. \label{LMUT}
\end{align}
By the estimates obtained from the existence and uniqueness theorems, it is not difficult to show that, for each $\eta,\zeta
\in W_2^{1,1}(\Pi_T)$
$$
\begin{aligned}
  - \iinte_{\Pi_T} \bfh^{m_i} \, \eta_t \, dz dt &\to - \iinte_{\Pi_T} \hat{\bfh} \, \eta_t \, dz dt,
\\
  \iinte_{\Pi_T} r \, h_{l z}^{m_i} \, \eta_z \, dz dt &\to \iinte_{\Pi_T} r \, \hat{h}_{l z} \, \eta_z \, dz dt, \quad
   l=1,2,
\\
  \iinte_{\Pi_T} u_t^{m_i} \, \zeta_t \, dz dt &\to \iinte_{\Pi_T} \hat{u}_t \, \zeta_t \, dz dt,
\\
  \iinte_{\Pi_T} (\nu^{m_i})^2 \, u_{z}^{m_i} \, \zeta_z \, dz dt &\to \iinte_{\Pi_T} \nu^2 \, \hat{u}_{z} \, \zeta_z \, dz dt.
\end{aligned}
$$
Moreover, if we proceed as in the proof of the existence theorem, then we can show that
$$
\begin{aligned}
  - \iinte_{\Pi_T} h^{m_i}_l \, u_t^{m_i} \, \eta_z \, dz dt &\to - \iinte_{\Pi_T} \hat{h}_l \, \hat{u} \, \eta_z \, dz dt, \quad  l=1,2,
\\
  \iinte_{\Pi_T} p \, ((\bfh^{m_i})^2)_z \, \zeta \, dz dt &\to\iinte_{\Pi_T} p \, (\hat{\bfh}^2)_z \, \zeta \, dz dt,
\end{aligned}
$$
for all $\eta,\zeta \in W_2^{1,1}(\Pi_T)$. Consequently, from the convergence above we conclude that $(\hat{\bfh}, \hat{u})$ is a weak solution weak the Cauchy problem \eqref{9}--\eqref{16} with initial data $\bfh_0$, $u_0$ and $u_1$. But, by uniqueness of the solution, we have
$$
  \bfh \equiv \hat{\bfh} \quad \text{and} \quad u \equiv \hat{u},
$$
which is a contradiction with our initial assumption.
\end{proof}

\bigskip

Besides we have proven the stability result, we give a more
refined result concerning the estimative. The result is
established in the case when $\nu$ is supposed to be a smooth
enough function.
\begin{theorem}\label{lem2}
Let $\bfh,u$ and $\bfh^m,u^m$, for each $m \in \nat$, be weak solutions of problems \eqref{9}--\eqref{16} and the associated
\eqref{50}--\eqref{53} respectively, and define
$$
   \bfv^m:=\bfh^m-\bfh, \qquad w^m:=u^m-u.
$$
Then, there exists a positive number $\delta$, independent of $\bfh,u,\bfh^m,u^m$, $t_1,t_2,$ such that for any $t_1,t_2\in[0,T]$, $0\leq t_2-t_1<\delta,$ the following inequality holds true
\begin{multline}\label{55}
   {\rm vrai}\max_{t\in[t_1,t_2]}
   \Big\{ p \, \|\bfv^m\|^2_{2,\T} +\|w^m_t\|^2_{2,\T} + 2 \, \|\nu w^m_z\|^2_{2,\T}\Big\} +2 \, p \, r_0 \, \|\bfv^m_{z}\|^2_{2,\Pi_{t_1t_2}}
   \\[5pt]
   \leq p \, C_{10} \{\|(r^m-r)\bfh_{z}\|^2_{2,\Pi_{t_1t_2}} + \|r^m \, \bfj^m - r \, \bfj\|^2_{2,\Pi_{t_1t_2}} \Big\}
   + 2 \, p \, \|\bfv^m(\cdot,t_1)\|^2_{2,\T}
   \\[5pt]
   + \|f^m-f\|^2_{2,\Pi_{t_1t_2}} + 2 \, \|w^m_t(t_1)\|^2_{2,\T} + 2 \, \|\nu \, w_z^m(t_1)\|^2_{2,\T}\,,
\end{multline}
where $\Pi_{t_1t_2}= (t_1,t_2) \times \T$, and $C_{10}$ is a positive constant independent of $t_1,t_2$.
\end{theorem}
Theorem \ref{lem2} is proved in much the same way as Lemma 4.2 from \cite{PV1}.
%%%%%%%%%%%%%%%%%%%%%%%%%%%%%%%%%%%%%%%%%%%%%%%%%%%%%%%%%%%%%%%%%%%%%%%%%%%%%%%%%%%%%%%%%%%%%%%%%%%%%%%%%%%%%%%%%
\section*{Acknowledgements}
%%%%%%%%%%%%%%%%%%%%%%%%%%%%%%%%%%%%%%%%%%%%%%%%%%%%%%%%%%%%%%%%%%%%%%%%%%%%%%%%%%%%%%%%%%%%%%%%%%%%%%%%%%%%%%%%%

The first author were partially supported by FAPERJ through the grant E-26/ 111.564/2008 entitled {\sl "Analysis, Geometry and Applications"}, and by the National Counsel of Technological and Scientific Development (CNPq) by the grant 311759/2006-8.

\bigskip

%%%%%%%%%%%%%%%%%%%%%%%%%%%%%%%%%%%%%%%%%%%%%%%%%%%%%%%%%%%%%%%%%%%%%%%%%%%%%%%%%%%%%%%%%%%%%%%%%%%%%%%%%%%%%%%%%%%%%%%%%%%%%%%%%%%%%%%%%%%%%%%%


\begin{thebibliography}{10}
%%%%%%%%%%%%%%%%%%%%%%%%%%%%%%%%%%%%%%%%%%%%%%%%%%%%%%%%%%%%%%%%%%%%%%%%%%%%%%%%%%%%%%%%%%%%%%%%%%%%%%%%%%%%%%%%%%%%%%%%%%%%%%%%%%%%%%%%%%%%%%%

\bibitem{AGSP} {\sc A.\,V.~Avdeev, E.\,V.~Goruynov,~ O.\,N.~Soboleva, ~and V.\,I.~Priimenko}, {\em Numerical solution of some direct and inverse problems of electromagnetoelasticity}, in J.~Inv.\ Ill-Posed Problems, 7(5) (1999), pp.~453--462.

\bibitem{BI} {\sc M.~Born, L. Infeld}, {\em Foundations of a new
field theory}, Proc. Roy. London, A, {\bf 144}, 1934,
pp.~425--451.

\bibitem{B} {\sc W.\,F.~Brown}, {\em Magnetoelasticity Interactions}, Sprin\-ger-Ver\-lag, New York---Berlin, 1966.

\bibitem{C} {\sc P.~Chadwick}, {\em Elastic wave propagation in magnetic field}, in Proc. of the IX-th Int. Congr. Appl. Mech., Brussels, 1956, pp.~143--156.

\bibitem{CD} {\sc B.D.~Coleman, E. Dill}, {\em Thermodynamic restrictions on the constitutive equations of
eletromagnetic theory}, Z. Angew. Math. Phys., {\bf 22}, 1971, pp.~691--702.

\bibitem{D} {\sc C. Dafermos}, {\em Hyperbolic conservation Laws in Continuum Physics}, Sprin\-ger-Ver\-lag, New York---Berlin, 2000.

\bibitem{DE} {\sc J.\,W.~Dunkin and A.\,C.~Eringen}, {\em On the propagation of waves in an electromagnetic elastic solid}, in Int. J. Eng. Sci., 1 (1963), pp.~461--495.

\bibitem{EM} {\sc A.\,C.~Eringen and G.\,A.~Maugin}, {\em Electrodynamics of Continua. Vols.~I,~II}, Sprin\-ger-Ver\-lag, New York---Berlin, 1990.

\bibitem{E} {\sc L.\,C.~Evans},  {\em Partial Differential Equations (Graduate Studies in Mathematics, Vol.~19)}, Amer. Math. Soc., Providence, RI, 2002.

\bibitem{IO} {\sc R.J.~Iório, V.M.~Iório}, {\em Fourier analysis and partial differential equations: An introduction}, Cambridge University Press, 2001.

\bibitem{VK} {\sc Vo-Khac Khoan}, {\em Distributions analyse de Fourier opérateurs aux dérivées partielles}, Tome 1--2, Lib. Vuibert, 63, Bd Saint-Germain, Paris, 1970.

\bibitem{K} {\sc L.~Knopoff}, {\em The interaction between elastic waves motion and a nagnetic field in electrical conducters}, in J. Geophys. Res., 60 (1955), pp.~441-456.

\bibitem{LSU} {\sc O.\,A.~Ladyzhenskaia,  V.\,A.~Solonnikov, and M.\,N.~Uralceva}, {\em Linear and Quasilinear Equations (Translations of Mathematical Monographs, Vol.23)}, Amer. Math. Soc., Providence, RI, 1968.

\bibitem{Li} {\sc J-L.~Lions}, {\em Quelques M\'{e}todes de R\'{e}solution des Probl\`{e}mes aux Limites non Lin\'{e}aires}, Dunod, Paris, 1969.

\bibitem{LoP} {\sc A.~Lorenzi and V.\,I.~Priimenko}, {\em Identification problems  related to elec\-tro-mag\-ne\-to-elas\-tic interactions}, in J.\ Inv.\ Ill-Posed Problems, 4 (1996), pp.~115--143.

\bibitem{LR} {\sc A.~Lorenzi and V.\,G.~Romanov}, {\em Identification of the electromagnetic coefficient connected with deformation currents},
in Inverse Problems, 9 (1993), pp.~301--319.

\bibitem{M} {\sc N.\,I.~Muskhelisvili}, {\em Some Basic Problems of Mathematical Theory of Elasticity}, Van Nostrand, Princeton, New Jersey, 1958.

\bibitem{NS} {\sc W. Neves and D. Serre}, {\em The incompleteness of the Born-Infeld model for non-linear multi-d Maxwell's equations}, Quart. Appl. Math. 63 (2005), pp.~343-367.

\bibitem{Pa} {\sc G.~Paria}, {\em Magneto-elasticity and magneto-thermo-elasticity}, in Adv. Appl. Mech., 10 (1967), pp.~73--112.

\bibitem{PV1} {\sc V.~Priimenko and M.~Vishnevskii}, {\em An initial boundary-value problem for model electromagnetoelasticity system}, in J. Differential Equations, 235 (2007), pp.~31–-55.

\bibitem{PV2} {\sc V.~Priimenko and M.~Vishnevskii}, {\em Chapter IV. Nonlinear Mathematical Problems of Electromagnetoelastic Interactions, pp.~ 99-155}, in Nonlinear Analysis Research Trends, Eds.: Ins R. Rouch, New York, Nova Science Publishers, Inc., 2008.

\bibitem{R1} {\sc V.\,G.~Romanov}, {\em On an inverse problem for a coupled system of equations of electrodynamics and elasticity}, in J.~Inv.\ Ill-Posed Problems, 3 (1995), pp.~321--332.

\bibitem{R2} {\sc V.\,G.~Romanov}, {\em Structure of a solution to the Cauchy problem for the system of the equations of electrodynamics and elasticity in the case of point sources}, in Siberian Math.~J., 36(3) (1995), pp.~541--561.

\bibitem{S} {\sc S.\,L.~Sobolev}, {\em Some Applications of Functional Analysis in Mathematical Physics}, Translations of Mathematical Monographs, Vol.90, Amer. Math. Soc., Providence, RI, 1991.

\bibitem{SW} {\sc E.M.~Stein, G.~Weiss}, {\em Introduction to Fourier analysis on euclidean spaces}, Princeton, New Jersey, Princeton University Press, Sixty Printing, 1990.

\end{thebibliography}
\end{document}